\documentclass[a4paper,12pt]{article}
\usepackage{ucs}
\usepackage{amsfonts, amssymb, amsmath, amsthm, amscd}
\usepackage{graphicx}
\usepackage{cite}
\ifx\undefined \pdfgentounicode \else
\input{glyphtounicode} \pdfgentounicode=1
\fi

\author{A.A. Vasil'eva\footnote{Faculty of Mechanics and Mathematics, Lomonosov Moscow State University; Moscow Center for Fundamental and Applied Mathematics.}}
\title{Widths and entropy numbers of embeddings of Sobolev classes on a H\"{o}lder domain\footnote{The paper was published with the financial support of  the Moscow Center for Fundamental and 
Applied Mathematics at Lomonosov Moscow State University under the agreement 075-15-2025-345.}}
\date{}
\begin{document}

\maketitle

\newenvironment{Biblio}{%
                  \renewcommand{\refname}{\footnotesize REFERENCES}%
                  \begin{thebibliography}{99}%
                  \itemsep -0.2em}%
                  {\end{thebibliography}}

\def\inff{\mathop{\smash\inf\vphantom\sup}}
\renewcommand{\le}{\leqslant}
\renewcommand{\ge}{\geqslant}
\newcommand{\sgn}{\mathrm {sgn}\,}
\newcommand{\inter}{\mathrm {int}\,}
\newcommand{\dist}{\mathrm {dist}}
\newcommand{\supp}{\mathrm {supp}\,}
\newcommand{\R}{\mathbb{R}}
\newcommand{\C}{\mathbb{C}}
\newcommand{\Z}{\mathbb{Z}}
\newcommand{\N}{\mathbb{N}}
\newcommand{\Q}{\mathbb{Q}}
\theoremstyle{plain}
\newtheorem{Trm}{Theorem}
\newtheorem{trma}{Theorem}
\newtheorem{Cor}{Corollary}
\newtheorem{Lem}{Lemma}
\newtheorem{Sta}{Proposition}

\theoremstyle{definition}
\newtheorem{Def}{Definition}
\newtheorem{Rem}{Remark}
\newtheorem{Sup}{Assumption}
\newtheorem{Supp}{Assumption}
\newtheorem{Not}{Notation}
\newtheorem{Exa}{Example}
\renewcommand{\proofname}{\bf Proof}
\renewcommand{\thetrma}{\Alph{trma}}
\renewcommand{\theSupp}{\Alph{Supp}}

\section{Introduction}

In \cite{besov_holder} upper estimates for the entropy numbers of Sobolev classes on a domain with H\"{o}lder boundary were obtained. In the present paper we improve these estimates (in the case when the definition of the Sobolev class involves all partial derivatives of order~$r$): we will show that the exponential in the orders of decrease  is greater than the quantity in  \cite{besov_holder}; moreover, under some conditions on the parameters, the order will be the same as for domains with Lipschitz boundary (while the boundary of the domain may be non-Lipschitz). We also obtain upper estimates for the Kolmogorov, linear and the Gelfand widths.

In addition, we will consider domains of special form, where the H\"older singularity is concentrated on an  $h$-set (see the definition below). For them, we will obtain more sharp upper estimates of widths and entropy numbers. The same lower estimates will be obtained in the case of special $h$-sets constructed in \cite{bricchi}. The last example shows that the exponent in the upper estimates on the class of all H\"{o}lder domains cannot be improved. On the other hand, it will be shown that these lower estimates  hold not for all  $h$-sets; for example, the order of decreae of entropy numbers and widths can be improved if a~plane is parallel to a~coordinate plane.

Let us give necessary definitions.

Let $d\in \N$, and let $\Omega\subset \R^d$ be a bounded domain. Given $f \in L_1^{{\rm loc}}(\Omega)$ and $\alpha = (\alpha_1, \, \dots, \, \alpha_d) \in \Z_+^d$, we denote $|\alpha| = \alpha_1 + \dots +\alpha_d$, $\partial_\alpha f = \frac{\partial^{|\alpha|} f}{\partial x_1^{\alpha_1} \dots \partial x_d^{\alpha_d}}$ (the partial derivative is taken in the sense of distributions). Given $r\in \Z_+$, we denote by $\nabla^r f$ the vector of all partial derivatives $\partial_\alpha f$ such that $|\alpha| = r$, and set
$$
\|\nabla^r f\| _{L_p(\Omega)} = \max _{|\alpha| = r} \|\partial_\alpha f\| _{L_p(\Omega)}.
$$

Let $1\le p\le \infty$, $r\in \N$, and let $\Omega \subset \R^d$ be a bounded domain. We define the Sobolev spaces by the formulas 
$$
{\cal W}^r_p(\Omega) = \{f\in L_1^{{\rm loc}} (\Omega) :\; \|\nabla^r f\|_{L_p(\Omega)} < \infty\},
$$
$$
\tilde{\cal W}^r_p(\Omega) = \{f\in L_1^{{\rm loc}} (\Omega) :\; \|\nabla^r f\|_{L_p(\Omega)} +\|f\|_{L_p(\Omega)} < \infty\}
$$
and the unit balls (the Sobolev classes), by the formulas
$$
W^r_p(\Omega) = \{f\in L_1^{{\rm loc}} (\Omega) :\; \|\nabla^r f\|_{L_p(\Omega)} \le 1\},
$$
$$
\tilde W^r_p(\Omega) = \{f\in L_1^{{\rm loc}} (\Omega) :\; \|\nabla^r f\|_{L_p(\Omega)} +\|f\|_{L_p(\Omega)} \le 1\}.
$$

It is well-known (see, e.g., \cite{sob_book}) that if the domain $\Omega$ is a finite unuion of bounded domains with Lipschitz boundary, then for $r + \frac dq -\frac dp >0$, for $r + \frac dq -\frac dp=0$, $q<\infty$, or for $r=d$, $p=1$, $q=\infty$, the set $\tilde W^r_p(\Omega)$ is bounded in $L_q(\Omega)$ (and $W^r_p(\Omega)$ is the sum of a bounded set and the space of polynomials of degree at most $d-1$). The same assertion holds if $\Omega$ is a John domain \cite{resh1, resh2, bojar} or satisfies a decaying flexible cone condition \cite{besov_rasp}. Embedding theorems for the Sobolev spaces on various irregular domains with zero angles (e.g., on H\"{o}lder domains, on domains with a flexible $\sigma$-cone condition an so on) were obtained in \cite{besov84, labutin1, labutin2, trushin, haj_kosk, kilp_mal, maz_pob, besov01, besov10, besov14, besov15, besov22, caso_ambr}.

Now we give the definitions of the Gelfand, Kolmogorov and linear widths.

Let $X$ be a linear normed space. Denote by $X^*$ the dual space to $X$, by ${\cal L}_n(X)$, the family of all linear subspaces in $X$ of dimension at most $n$, and by $L(X, \, Y)$, the space of linear continuous operators $A:X\to Y$. Given $A\in L(X, \, X)$, we denote by ${\rm rk}\, A$ the dimension of the range of $A$. Let $M\subset X$ be a non-empty set, $n\in \Z_+$. The Kolmogorov, linear and the Gelfand $n$-widths of $M$ in $X$ are defined, respectively, by the formulas 
$$
d_n(M, \, X) = \inf _{L\in {\cal L}_n(X)} \sup _{x\in M} \inf
_{y\in L} \|x-y\|,
$$
$$
\lambda_n(M, \, X) = \inf _{A\in L(X, \, X), \, {\rm rk}\, A\le n} \sup _{x\in M} \|x-Ax\|,
$$
$$
d^n(M, \, X) = \inf _{x_1^*, \, \dots, \, x_n^*\in X^*} \sup \Bigl\{\|x\|:\; x\in M\cap \Bigl(\cap _{j=1}^n \ker x_j^*\Bigr)\Bigr\}
$$
(in the last definition, we set $\sup \varnothing = 0$).

The problem of estimating the widhts of Sobolev classes on one-dimensional and multi-dimensional Lipschitz domains was studied in \cite{tikh60, babenko, mityagin, makovoz, ismagilov, majorov, bib_kashin, kashin_sma, kulanin1, kulanin2, galeev85, galeev87, teml3, teml4, mal25} (notice that the paper \cite{mal25} is a recent result, which completely solves the Sobolev width problem on an interval; the critical case of $W^1_1$ in $L_q$, $2<q<\infty$, is considered). For John domains the order estimates of the widths were obtained in \cite{besov_cusp} and \cite{vas_john}; these estimates are the same as for a cube. In \cite{vas_john} a more general case of weighted spaces was considered, and the proof was more complicated than in \cite{besov_cusp}. For $r = 1$, $p = q$ and irregular domains (with zero angles) the estimates for linear widths were obtained by W.D. Evans, D.J. Harris and Y. Saito in \cite{evans_harris, evans_harris2} (more precisely, in these papers the approximation numbers were considered, but in the case of compact embeddings they are equal to the linear widths \cite{heinr}). In the case $p<q$, for a domain with a singular cusp of the form $$\{(x', \, x_d)\in \R^{d-1}\times \R:\, 0< x_d< 1, \, |x'| < x_d^\sigma\}$$ the order estimates for the Kolmogorov widths were obtained in \cite{besov_cusp} (here $\sigma > 1$, $r+(\sigma(d-1)+1)(1/q-1/p)>0$); the orders were found to be the same as for the cube. In \cite{vas_cusp} the domains $$\{(x', \, x_d)\in \R^{d-1}\times \R:\, 0< x_d< 1/2, \, |x'| < x_d^\sigma|\log x_d|^\alpha\}$$ were considered, where $r = (\sigma(d-1)+1)(1/p-1/q)$; in this case, the orders may change already for small $\alpha>0$. 

Now we give the definition of the entropy numbers.

Let $X$, $Y$ be normed spaces, and let $T\in L(X,\, Y)$, $n\in \N$. The entropy numbers of the operator $T$ are defined by
$$
e_n(T) = \inf \Bigl\{ \varepsilon>0:\; \exists y_1, \, \dots, \, y_{2^{n-1}}\in Y:\; T(B_X) \subset \cup _{i=1}^{2^{n-1}}(y_i+\varepsilon B_Y)\Bigr\}.
$$

If $\Omega$ is a cube, the order estimates for the entropy numbers of the embedding operator ${\rm Id}:\tilde {\cal W}^r_p(\Omega)\to L_q(\Omega)$ were obtained in \cite{bir_sol}, \cite{edm_tr}. In \cite{bir_sol} the proof depends on an approximation of a function  $f\in \tilde {\cal W}^r_p(\Omega)$ by 
a piecewise polynomial function with an appropriate partition. In \cite{edm_tr} the discretization method and Sch\"{u}tt's theorem \cite{schutt} about estimates for the entropy numbers of finite-dimensional balls were applied. In \cite{vas_entr} estimates for the entropy numbers were obtained for John domains; they are the same as for a cube (again, here the more general case of weighted spaces was considered; in the non-weighted case the proof can be significantly simplified). In \cite{besov_holder} upper estimates for the entropy numbers were obtained in the case of H\"{o}lder domains and their anisotropic generalizations. Here these estimates will be improved.

We also mention the paper \cite{edm89}, where the upper estimates for the entropy and approximation numbers of the embedding operator of the anisotropic Sobolev space on a domain with a horn condition into the Orlicz space were obtained.

Now we formilate the main results.

Let the functions $\varphi_i:(0, \, 1] \rightarrow (0, \, 1]$ $(1\le i\le d-1)$ satisfy the following conditions:
\begin{align}
\label{phi_prop} \varphi_i(t)\le a_*t, \quad \varphi_i \uparrow, \quad \varphi_i(2t)\le a_* \cdot \varphi_i(t), \quad 0\le t\le 1, \; 1\le i\le d-1,
\end{align}
where $a_*\ge 1$.

\begin{Def}
Let $\Omega \subset \R^d$ be a domain. We write $\Omega\in {\cal G}_{\varphi_1, \, \dots, \, \varphi_{d-1}}$ if
\begin{align}
\label{omega_def}
\Omega = \{(x', \, x_d):\; x'\in (0, \, 1)^{d-1}, \, 0<x_d<\psi(x')\},
\end{align}
where $\psi:[0, \, 1]^{d-1} \rightarrow (0, \, \infty)$,
\begin{align}
\label{psi_prop} \begin{array}{c} |\psi(x')-\psi(y')| \le t \quad \text{if }|x_i-y_i| \le \varphi_i(t), \quad i=1, \, \dots, \, d-1; \\ \min _{x'\in [0, \, 1]^{d-1}} \psi(x')\in [1, \, 2].\end{array}
\end{align}
We write $\Omega\in {\cal G}'_{\varphi_1, \, \dots, \, \varphi_{d-1}}$ if $\Omega = \cup _{j=1}^{j_0} \Omega_j$, $\Omega_j = T_j(G_j)$, where $G_j \in {\cal G}_{\varphi_1, \, \dots, \, \varphi_{d-1}}$, and $T_j$ is a composition of a homothety, an orthogonal operator and a translation.
\end{Def}

In particular, a finite union of H\"{o}lder domains of order $1/\sigma$ belongs to ${\cal G}'_{\varphi_1, \, \dots, \, \varphi_{d-1}}$ with $\varphi_i(t) = Ct^\sigma$, $i=1, \, \dots, \, d$ (here $\sigma \ge 1$).

We introduce notation for order inequalities and equalities. Let $X$, $Y$ be sets, and let $f_1$, $f_2:\ X\times Y\rightarrow \mathbb{R}_+$. We write $f_1(x, \, y)\underset{y}{\lesssim} f_2(x, \, y)$ (or
$f_2(x, \, y)\underset{y}{\gtrsim} f_1(x, \, y)$), if for any
$y\in Y$ there is $c(y)>0$ such that $f_1(x, \, y)\le
c(y)f_2(x, \, y)$ for each $x\in X$; $f_1(x, \,
y)\underset{y}{\asymp} f_2(x, \, y)$ if $f_1(x, \, y)
\underset{y}{\lesssim} f_2(x, \, y)$ and $f_2(x, \,
y)\underset{y}{\lesssim} f_1(x, \, y)$.

We say that a function is locally absolutely continuous on a convex subset of $\R$ if it is absolutely continuous on each segment lying in this set.

Let $u:[1, \, \infty) \rightarrow (0, \, \infty)$ be a locally absolutely continuous function such that $\lim \limits _{t\to +\infty} t \frac{u'(t)}{u(t)} = 0$, $\theta>0$. Then (see \cite[Lemma 2]{vas_width}) for sufficiently large $s\ge 1$ the equation $t^\theta u(t) = s$ has a unique solution $t(s)$, and $t(s) \underset{\theta,u}{\asymp} s^{1/\theta} \varphi _{\theta,u}(s)$, where $\varphi _{\theta,u}$ is a locally absolutely continuous function such that $\lim \limits _{s\to +\infty} s \frac{\varphi _{\theta,u}'(s)}{\varphi _{\theta,u}(s)} = 0$.

Let $\Lambda:(0, \, 1] \rightarrow (0, \, \infty)$ be a locally absolutely continuous function such that 
\begin{align}
\label{lam_slow} \lim \limits _{t\to +0} t\frac{\Lambda'(t)}{\Lambda(t)} = 0.
\end{align}
We set $\psi_\Lambda(t) = \frac{1}{\Lambda(1/t)}$, $t\in [1, \, \infty)$. Then $\lim \limits _{t\to +\infty} t\frac{\psi_{\Lambda}'(t)}{\psi_\Lambda(t)} = 0$.

Denote $\mathfrak{Z}_* = (p, \, q, \, r, \, d, \, \Omega)$.

\begin{Trm}
\label{main_entr} Let $1\le p\le q\le\infty$, $\Omega \in {\cal G}'_{\varphi_1, \, \dots, \, \varphi_{d-1}}$, where $\varphi_i$ satisfy \eqref{phi_prop}. Suppose that 
\begin{align}
\label{prod_phi_i}
\prod _{i=1}^{d-1} \varphi_i(t) = t^{\sigma(d-1)}\Lambda(t),
\end{align}
where $\sigma \ge 1$ and the function $\Lambda$ is locally absolutely continuous and satisfies \eqref{lam_slow}. Let $r\in \N$, 
\begin{align}
\label{emb_hol_cond}
r + (\sigma(d-1)+1)(1/q-1/p) > 0.
\end{align}
We set
\begin{align}
\label{a1a2def}
\alpha_1=\frac rd, \quad\alpha_2=\frac{r + 1/q-1/p}{\sigma(d-1)},
\end{align}
\begin{align}
\label{t1t2def}
\tau_1(n)=1, \quad \tau_2(n) = (\varphi_{\sigma(d-1),\psi_\Lambda}(n))^{-r-(\sigma(d-1)+1)(1/q-1/p)} [\psi_\Lambda(n^{\frac{1}{\sigma(d-1)}}\varphi_{\sigma(d-1),\psi_\Lambda}(n))]^{1/p-1/q}.
\end{align}
Let $\alpha_1\ne \alpha_2$, $j_*\in \{1, \, 2\}$, $\alpha_{j_*}= \min \{\alpha_1, \, \alpha_2\}$. Then
\begin{align}
\label{e_n_est}
e_n({\rm Id}:\tilde {\cal W}^r_p(\Omega)\rightarrow L_q(\Omega)) \underset{\mathfrak{Z}_*}{\lesssim} n^{-\alpha_{j_*}} \tau_{j_*}(n),
\end{align}
where ${\rm Id}$ is the embedding operator.
\end{Trm}

In \cite{besov_holder} the case of $\varphi_i(t) = t^{\lambda_i}$, where $\lambda_i\ge 1$ ($i=1, \, \dots, \, d-1$), $\lambda_d:=1$ was considered. Then $\sigma(d-1)+1 = \sum \limits _{i=1}^d \lambda_i$. The upper estimate for $n$th entropy number was given by the formula $n^{-\frac{r}{\sigma(d-1)+1}}$. In \S 7 it will be proved that each domain from \cite{besov_holder} is a finite union of domains from ${\cal G}'_{\varphi_1, \, \dots, \, \varphi_{d-1}}$ and domains satisfying the cone condition. Since $\frac{r}{\sigma(d-1)+1} < \frac{r+1/q-1/p}{\sigma(d-1)}$ by \eqref{emb_hol_cond}, Theorem \ref{main_entr} improves the estimates from \cite{besov_holder}. We notice that for a special subclass of H\"{o}lder domains in \cite{besov_holder} the estimates of the entropy numbers were obtained for wider Sobolev classes defined by conditions in  the  $L_p$-norm on the function and its derivatives 
 $\frac{\partial ^r}{\partial x_i^r}$, $1\le i\le d$, with no  conditions on the mixed derivatives.

We denote by $\vartheta_n(W^r_p(\Omega), \, L_q(\Omega))$ the Kolmogorov, linear or Gelfand width of the set $W^r_p(\Omega)$ in $L_q(\Omega)$; by $\hat q$ we denote, respectively, $q$, $\min \{q, \, p'\}$ and $p'$. A similar notation is used if  $W^r_p(\Omega)$ is replaced by $\tilde W^r_p(\Omega)$.

\begin{Trm}
\label{main_widths} Suppose that the conditions of Theorem {\rm \ref{main_entr}} hold, $1<p\le q<\infty$, $\alpha_1$ and $\alpha_2$ are defined by \eqref{a1a2def}, $\tau_1$ and $\tau_2$ are defined by \eqref{t1t2def}.
\begin{enumerate}
\item Let $\hat q\le 2$ or $p = q$, and let $\alpha_1\ne \alpha_2$. Then
$$
\vartheta_n(W^r_p(\Omega), \, L_q(\Omega)) \underset{\mathfrak{Z}_*}{\lesssim} n^{-\alpha_{j_*}+1/p-1/q} \tau_{j_*}(n),
$$
where $j_*$ is defined by the condition $\alpha_{j_*} = \min \{\alpha_1, \, \alpha_2\}$.

\item Let $p<q$, $\hat q>2$. We set
$$
\theta_1 = \frac rd + \frac 1q-\frac 1p + \min \Bigl\{ \frac 1p-\frac 1q, \, \frac 12-\frac{1}{\hat q}\Bigr\}, \quad \theta_2 = \frac{\hat q}{2} \Bigl( \frac rd+\frac 1q-\frac 1p\Bigr),
$$
$$
\theta_3 = \frac{r + (\sigma(d-1)+1)(1/q-1/p)}{\sigma(d-1)} + \min \Bigl\{ \frac 1p-\frac 1q, \, \frac 12-\frac{1}{\hat q}\Bigr\},
$$
$$
\theta_4 = \frac{\hat q}{2} \cdot \frac{r + (\sigma(d-1)+1)(1/q-1/p)}{\sigma(d-1)},
$$
$\tilde \tau_1(n)=\tilde \tau_2(n) = 1$, $\tilde \tau_3(n) = \tau_2(n)$, $\tilde \tau_4(n) = \tau_2(n^{\hat q/2})$.
Suppose that there is $j_*\in \{1, \, 2, \, 3, \, 4\}$ such that
$\theta_{j_*}< \min _{j\ne j_*} \theta_j$. Then
\begin{align}
\label{poper_est_main}
\vartheta_n(W^r_p(\Omega), \, L_q(\Omega)) \underset{\mathfrak{Z}}{\lesssim} n^{-\theta_{j_*}} \tau_{j_*}(n).
\end{align}
\end{enumerate}
\end{Trm}

\begin{Rem}
From the inclusion $\tilde W^r_p(\Omega) \subset W^r_p(\Omega)$ it follows that the same upper estimates hold for $\vartheta_n(\tilde W^r_p(\Omega), \, L_q(\Omega))$.
\end{Rem}

If $\frac rd < \frac{r+1/q-1/p}{\sigma(d-1)}$, then the upper estimate for the entropy numbers and widths is the same as for $\Omega = (0, \, 1)^d$. In can be proved by a standard method that in this case the same lower estimate holds.

Now we consider the special case of domains $\Omega$ of form \eqref{omega_def}, when $\psi$ is a function of a distance to an $h$-set.

By $B_t(x)$ we denote the ball in $(\R^k, \, \|\cdot\|_{l_\infty^k})$ of radius $t$ centered at the point $x$, where $\|(x_1, \, \dots, \, x_k)\|_{l_\infty^k} = \max _{1\le i\le k} |x_i|$.

\begin{Def}
\label{def_h_set}
Let $\Gamma \subset \R^k$ be a non-empty compact set, and let $h:(0, \, 1] \rightarrow (0, \, \infty)$ be a non-decreasing function. We say that $\Gamma$ is an $h$-set if there are a constant $c_*\ge 1$ and a finite $\sigma$-additive Borel measure $\mu$ on $\R^k$ such that ${\rm supp}\, \mu = \Gamma$ and 
\begin{align}
\label{h_set_def} c_*^{-1} h(t) \le \mu(B_t(x)) \le c_* h(t), \quad x\in \Gamma, \quad t\in (0, \, 1].
\end{align}
\end{Def}

\begin{Rem}
If an $h$-set exists, then the function $h$ satisfies the doubling condition, i.e., $h(2t)\le b_*h(t)$, $0<t\le \frac 12$, where $b_*>0$ does not depend on $t$.
Hence if we consider balls with respect to an arbitrary norm on $\R^k$, the condition \eqref{h_set_def} will be true with some other constant $c_*$.
\end{Rem}

Lipschitz $l$-dimensional surfaces ($l\in \{0, \, 1, \, \dots, \, k-1\}$; here $h(t)=t^l$), the Koch's curve (see \cite[pp. 66--68]{mattila}) and some Cantor-type sets (see \cite{bricchi}) are examples of $h$-sets.

For a non-empty set $A\subset \R^k$ and $x\in \R^k$ we denote
$$
{\rm dist}(x, \, A) = \inf _{y\in A}\|x-y\|_{l_\infty ^k}.
$$

\begin{Def}
Let $0\le \theta<d$, $\sigma\ge 1$. We write $\Omega \in {\cal G}_{\theta,\sigma}$ if $\Omega$ is given by formula \eqref{omega_def} with $\psi(x') = 2 - ({\rm dist}(x', \, \Gamma))^{1/\sigma}$, where $\Gamma \subset [0, \, 1]^{d-1}$ is an $h$-set with $h(t) = t^\theta$.
\end{Def}

\begin{Rem}
It is easy to check that the function $\psi$ satisfies \eqref{psi_prop} with $\varphi_1(t) = \dots = \varphi_{d-1}(t) = at^\sigma$ for some $a=a(\sigma)>0$. Hence each domain from the class ${\cal G}_{\theta,\sigma}$ belongs to the class ${\cal G}_{\varphi_1, \, \dots, \, \varphi_{d-1}}$.
\end{Rem}

Notice that the magnitude $\alpha_2$ from \eqref{a1a2def} can be written as
$$
\alpha_2 = \frac{r + (1/q-1/p)(\sigma(d-1)+1)}{\sigma(d-1)} + \frac 1p - \frac 1q.
$$

\begin{Trm}
\label{h_set_e_w} Let $0\le \theta <d$, $\sigma \ge 1$, $\Omega \in {\cal G}_{\theta,\sigma}$, and let \eqref{emb_hol_cond} hold. Then the analogue of the assertions of Theorems {\rm \ref{main_entr}}, {\rm \ref{main_widths}} holds, where $\frac{r + (1/q-1/p)(\sigma(d-1)+1)}{\sigma(d-1)}$ is replaced by $\frac{r + (1/q-1/p)(\sigma(d-1)+1)}{\sigma\theta}$.
\end{Trm}

So, this increases the absolute value of the exponential in the upper estimate 
for a function  $\psi$ of special kind.

We will also show that if $\Gamma$ is the set constructed in \cite{bricchi}, then under conditions of Theorem \ref{h_set_e_w} the similar lower estimate for the entropy numbers and widths holds. Since $\theta$ can be taken arbitrarily close to $d$, the upper estimates in Theorems \ref{main_entr}, \ref{main_widths} cannot be improved (i.e., we cannot replace the magnitude $\frac{r + (1/q-1/p)(\sigma(d-1)+1)}{\sigma(d-1)}$ by a greater one so that the assertions of these theorems hold on the whole class of the domains $\Omega \in {\cal G}'_{\varphi_1, \, \dots, \, \varphi_{d-1}}$ with $\varphi_i(t) = t^\sigma$, $i=1, \, \dots, \, d-1$).

\begin{Rem}
If $h(t) = t^\theta R(t)$, $\psi(x') = 2 - ({\rm dist}(x', \, \Gamma))^{1/\sigma}W({\rm dist}(x', \, \Gamma))$, where $R$, $W$ are locally absolutely continuous functions on $(0, \, 1]$ and $\lim \limits _{t\to +0} t\frac{R'(t)}{R(t)} = \lim \limits _{t\to +0} t\frac{W'(t)}{W(t)}=0$, and the function $\varphi_0$ inverse to $\psi_0(t):=t^{1/\sigma}W(t)$ in a right neighborhood of 0 satisfies the inequality $\varphi_0(t) \le a_*t$, it is also possible to obtain the estimates for the entropy numbers and widths (see Theorems \ref{trm1}, \ref{trm2} in \S 2), but the formulas would be rather complicated. So we consider the more simple case, when $R(\cdot)\equiv 1$ and $W(\cdot) \equiv 1$.
\end{Rem}

\begin{Exa}
\label{exa1}
Let $\theta\in \{1, \, \dots, \, d-2\}$,
$$
\Gamma = \{(x_1, \, \dots, \, x_\theta, \, 1/2, \, \dots, \, 1/2):\; 0\le x_i\le 1, \, i=1, \, \dots, \, \theta\}.
$$
It will be shown that under the conditions of Theorem \ref{h_set_e_w} in the upper estimates from Theorems \ref{main_entr}, \ref{main_widths} the magnitude $\frac{r + (1/q-1/p)(\sigma(d-1)+1)}{\sigma(d-1)}$ can be replaced by $$\frac{r + (1/q-1/p)(\sigma(d-\theta-1)+\theta+1)}{\theta}.$$ Since $$\frac{r + (1/q-1/p)(\sigma(d-\theta-1)+\theta+1)}{\theta} > \frac{r + (1/q-1/p)(\sigma(d-1)+1)}{\sigma\theta},$$
we obtain an example of an $h$-set (the plane parallel to the coordinate one), for which the estimates in Theorem \ref{h_set_e_w} can be improved. Thus, the estimates of the widths and the entropy numbers depend not only on the dimension of the $h$-set of singularities of the function $\psi$, but also depend on its structure.
\end{Exa}

\section{Auxilliary assertions}

We formulate theorems about upper estimates of the widths and the entropy numbers of embeddings of function classes on sets with tree-like structure (see \cite{vas_entr, vas_width}). Here for brevity we consider only Sobolev and Lebesgue spaces on a domain in $\R^d$.

Let $({\cal T}, \, \xi_0)$ be a tree with a distinguished vertex $\xi_0$. Then the partial order and a distance $\rho$ on its vertex set ${\bf V}({\cal T})$ can be defined naturally (here $\xi_0$ is the minimal vertex of the tree). Given $\xi \in {\bf V}({\cal T})$, $j\in \Z_+$, we set
$$
{\bf V}_j(\xi) = \{\eta\ge \xi:\; \rho(\xi, \, \eta) = j\}
$$
and denote by ${\cal T}_\xi$ the subtree in ${\cal T}$ with the vertex set $\{\eta \in {\bf V}({\cal T}):\; \eta \ge \xi\}$.

The Lebesgue measure of a measurable subset $E\subset \R^d$ will be denoted by ${\rm mes}\, E$ or $|E|$.

Let $m\in \N$ or $m=\infty$, $E\subset \R^d$, $E_j\subset \R^d$ $(j=1, \, 2, \, \dots, \, m)$ be measurable subsets. We say that $\{E_j\}_{j=1}^m$ is a partition of $E$ if ${\rm mes}(E_i\cap E_j)=0$ for $i\ne j$ and ${\rm mes}\Bigl(E \bigtriangleup \Bigl(\cup _{j=1} ^m E_j\Bigr)\Bigr)=0$.

Let $\Omega \subset \R^d$ be a bounded domain, let $\hat \Theta$ be at most countable partition of $\Omega$ into subdomains, let $({\cal A}, \, \xi_*)$ be a tree such that
\begin{align}
\label{vert_c} \exists c_1\ge 1:\quad {\rm card}\, {\bf V}_1(\xi)\le c_1, \quad \xi \in {\bf V}({\cal A}),
\end{align}
and let $\hat F:{\bf V}({\cal A}) \rightarrow \hat \Theta$ be a bijection.

Denote by ${\cal P}_{r-1}(\Omega)$ the space of polynomials on $\Omega$ of degree at most $r-1$. Given a measurable subset $E\subset \Omega$, we set ${\cal P}_{r-1}(E) = \{f|_{E}:\; f\in {\cal P}_{r-1}(\Omega)\}$.

For each subtree ${\cal A}'\subset {\cal A}$ we denote
\begin{align}
\label{omacup} \Omega^*_{{\cal A}'}=\cup _{\xi\in {\bf V}({\cal
A}')} \hat F(\xi).
\end{align}

\begin{Sup}
\label{sup1} There is a function $w_*:{\bf V}({\cal A})\rightarrow (0, \, \infty)$ with the following property: for each vertex $\hat\xi\in {\bf V}({\cal A})$ there is a linear continuous projection
$P_{\hat\xi}:L_q(\Omega)\rightarrow {\cal P}_{r-1}(\Omega)$ such that for any function $f\in {\cal W}^r_p(\Omega)$ and each subtree ${\cal
A}'\subset {\cal A}$ with minimal vertex $\hat\xi$
\begin{align}
\label{f_pom_f} \|f-P_{\hat\xi}f\|_{L_q(\Omega^*_{{\cal A}'})}\le
w_*(\hat\xi)\|\nabla^r f\| _{L_p(\Omega^*_{{\cal A}'})}.
\end{align}
\end{Sup}

\begin{Sup}
\label{sup2} There are numbers $\delta_*>0$, $c_2\ge 1$
such that for each vertex $\xi\in {\bf V}({\cal A})$ and each
$n\in \N$, $m\in \Z_+$ there is a partition $T_{m,n}(G)$ of the set $G=\hat F(\xi)$ with the following properties:
\begin{enumerate}
\item ${\rm card}\, T_{m,n}(G)\le c_2\cdot 2^mn$.
\item For each $E\in T_{m,n}(G)$ there is a linear continuous
operator $P_E:L_q(\Omega)\rightarrow {\cal P}_{r-1}(E)$ such that for any function $f\in {\cal W}^r_p(\Omega)$
$$
\|f-P_Ef\|_{L_q(E)}\le (2^mn)^{-\delta_*}w_*(\xi) \|\nabla^r f\|
_{L_p(E)}.
$$
\item For each $E\in T_{m,n}(G)$
$$
{\rm card}\,\{E'\in T_{m\pm 1,n}(G):\, {\rm mes}(E\cap E') >0\}
\le c_2.
$$
\end{enumerate}
\end{Sup}

\begin{Sup}
\label{sup3} There are numbers $\lambda_*\ge 0$,
$\gamma_*>0$, locally absolutely continuous functions $u_*:(0, \, \infty)
\rightarrow (0, \, \infty)$ and $\psi_*:(0, \, \infty) \rightarrow
(0, \, \infty)$, numbers $c_3\ge 1$, $t_0\in \Z_+$ and a partition $\{{\cal A}_{t,i}\}_{t\ge t_0, \, i\in \hat J_t}$ of the tree ${\cal A}$ such that $\lim \limits _{y\to \infty}
\frac{yu_*'(y)}{u_*(y)}=0$, $\lim \limits _{y\to \infty}
\frac{y\psi_*'(y)}{\psi_*(y)}=0$,
\begin{align}
\label{w_s_2} c_3^{-1} 2^{-\lambda_*t}u_*(2^{t}) \le
w_*(\xi)\le c_3\cdot 2^{-\lambda_*t}u_*(2^{t}), \quad \xi
\in {\bf V}({\cal A}_{t,i}),
\end{align}
\begin{align}
\label{nu_t_k} \sum \limits _{i\in \hat J_t} {\rm
card}\, {\bf V}({\cal A}_{t,i})\le c_3\cdot 2^{\gamma_*t}
\psi_*(2^{t}), \quad t\ge t_0.
\end{align}

In addition, we suppose that the following properties hold.
\begin{enumerate}

\item Let $t$, $t'\in \Z_+$. Then
\begin{align}
\label{2l} 2^{-\lambda_*t'}u_*(2^{t'})\le c_3\cdot
2^{-\lambda_*t}u_*(2^{t}) \quad\text{for}\quad  t'\ge t.
\end{align}

\item If the tree ${\cal A}_{t',i'}$ follows the tree ${\cal A}_{t,i}$ (i.e., the minimal vertex of ${\cal A}_{t',i'}$ follows some vertex of ${\cal A}_{t,i}$), then $t'=t+1$.

\item ${\rm card}\{i':\; {\cal A}_{t',i'} \text{ follows the tree } {\cal A}_{t,i}\} \le c_3$.
\end{enumerate}
\end{Sup}

Let $$\hat {\cal W}^r_p(\Omega) =\{f - P_{\xi_0}f:\, f\in {\cal W}^r_p(\Omega)\}.$$

We write $\mathfrak{Z}_0=(p, \, q, \, r, \, d, \, \delta_*, \, \lambda_*, \, \gamma_*, \, u_*, \, \psi_*, \, c_1, \, c_2, \, c_3)$.

The following theorem was proved in \cite{vas_entr} for $p>1$, $q<\infty$, but these conditions were not applied in the proof. Thus, we can formulate this theorem for $1\le p\le q\le \infty$.

\begin{trma}
\label{trm1} {\rm (see \cite{vas_entr}).} Let $1\le p\le q\le \infty$, and let Assumptions {\rm \ref{sup1}}, {\rm \ref{sup2}}, {\rm \ref{sup3}} hold. Let $\delta_*>0$, $\delta_*\ne \lambda_*/\gamma_*$. We set
$$
\sigma_*(n)=\left\{ \begin{array}{l} 1 \quad \text{for}\quad
\delta_*<\lambda_*/\gamma_*, \\ u_*(n^{1/\gamma_*} \varphi_{\gamma_*,\psi_*}(n))
\varphi_{\gamma_*,\psi_*}^{-\lambda_*}(n) \quad \text{for}\quad \delta_* >
\lambda_*/\gamma_*.\end{array}\right.
$$
Then
$$
e_n({\rm Id}:\hat {\cal W}^r_p(\Omega)\rightarrow L_q(\Omega))
\underset{\mathfrak{Z}_0} {\lesssim} n^{-\min(\delta_*, \,
\lambda_*/\gamma_*)+\frac 1q-\frac 1p}\sigma_*(n).
$$
\end{trma}

The following result was proved in \cite{vas_width}.
\begin{trma}
\label{trm2} {\rm (see \cite{vas_width}).}
Suppose that the conditions of Theorem {\rm \ref{trm1}} hold, $p>1$, $q<\infty$.
\begin{enumerate}
\item Let $\hat q\le 2$ or $p=q$. Let the function $\sigma_*$ be defined as in Theorem {\rm \ref{trm1}}. Then
$$
\vartheta_n(W^r_p(\Omega), \, L_q(\Omega)) \underset{\mathfrak{Z}_0} {\lesssim} n^{-\min(\delta_*, \,
\lambda_*/\gamma_*)}\sigma_*(n).
$$

\item Let $\hat q>2$, $p<q$. We set
$$
\theta_1 = \delta_* + \min \{1/p-1/q, \, 1/2-1/\hat q\}, \quad \theta_2 = \frac{\hat q \delta_*}{2},
$$
$$
\theta_3 = \lambda_*/\gamma_* + \min \{1/p-1/q, \, 1/2-1/\hat q\}, \quad \theta_4 = \frac{\hat q \lambda_*}{2\gamma_*},
$$
$\tau_1(n)=\tau_2(n)=1$, $$\tau_3(n) = u_*(n^{1/\gamma_*} \varphi_*(n))
\varphi_*^{-\lambda_*}(n),$$ $$\tau_4(n) = u_*(n^{\hat q/2\gamma_*} \varphi_*(n^{\hat q/2}))
\varphi_*^{-\lambda_*}(n^{\hat q/2}).$$ Suppose that there exists $j_*\in \{1, \, 2, \, 3, \, 4\}$ such that $\theta_{j_*} < \min _{j\ne j_*}\theta_j$. Then
$$
\vartheta_n(W^r_p(\Omega), \, L_q(\Omega)) \underset{\mathfrak{Z}_0} {\lesssim} n^{-\theta_{j_*}}\tau _{j_*}(n).
$$

\end{enumerate}
\end{trma}

Let $({\cal T}, \xi_0)$ be a tree, and let $g$, $v:{\bf V}({\cal
T}) \rightarrow \R_+$. We define the weighted summation operator $S_{g,v,{\cal T}}$ by formula
$$
S_{g,v,{\cal T}}f(\xi) = v(\xi)\sum \limits _{\xi'\le \xi}
g(\xi')f(\xi'), \quad \xi \in {\bf V}({\cal T}), \quad f:{\bf
V}({\cal T}) \rightarrow \R.
$$
Let $1\le p\le q\le \infty$. We denote by
$\mathfrak{S}^{p,q}_{{\cal T},g,v}$ the minimal constant $C$ in the inequality
$$
\left(\sum \limits_{\xi \in {\bf V}({\cal T})} |S_{g,v,{\cal T}}f(\xi)|^q\right)^{1/q}
\le C\left(\sum \limits_{\xi \in {\bf V}({\cal T})}
|f(\xi)|^p\right)^{1/p}, \;\; f:{\bf V}({\cal T})\rightarrow \R
$$
(naturally modified for $p=\infty$ or $q=\infty$).

The sharp two-sided estimates for norms of such operators were obtained in \cite{a_s}. If $q<\infty$ and there are numbers $a>0$, $b>0$ such that, for any $\xi \in {\bf V}({\cal T})$, $j\in \Z_+$,
\begin{align}
\label{v_w_cond} \sum \limits _{\eta \in {\bf V}_j(\xi)} v^q(\eta) \le b\cdot 2^{-aj} v^q(\xi),
\end{align}
it is easy to check that
\begin{align}
\label{sum_est_0} \mathfrak{S}^{p,q}_{{\cal T},g,v} \underset{a,b}{\lesssim} \sup _{\xi \in {\bf V}({\cal T})} g(\xi)v(\xi).
\end{align}
Indeed, let $g_1(\xi) = 2^{\varepsilon l}$, $g_2(\xi) = 2^{-\varepsilon l}g(\xi)$, $\xi\in {\bf V}_l(\xi_0)$. Then, if $\varepsilon = \varepsilon(a)>0$ is sufficiently small, we have
$$
\Bigl( \sum \limits _{\xi \in {\bf V}({\cal T})}v^q(\xi) \Bigl|\sum \limits _{\xi_0\le \eta \le \xi} g(\eta) f(\eta)\Bigr|^q\Bigr)^{1/q}\le
$$ 
$$
\le \Bigl( \sum \limits _{\xi \in {\bf V}({\cal T})}v^q(\xi) \Bigl(\sum \limits _{\xi_0\le \eta \le \xi} g_1^{q'}(\eta)\Bigr)^{q/q'}\sum \limits _{\xi_0\le \zeta \le \xi} g_2^q(\zeta) f^q(\zeta)\Bigr)^{1/q} \underset{\varepsilon}{\lesssim}
$$
$$
\lesssim \Bigl( \sum \limits _{\xi \in {\bf V}({\cal T})}v^q(\xi) g_1^q(\xi)\sum \limits _{\xi_0\le \zeta \le \xi} g_2^q(\zeta) f^q(\zeta)\Bigr)^{1/q}=
$$
$$
= \Bigl( \sum \limits _{\zeta \in {\bf V}({\cal T})}g_2^q(\zeta) f^q(\zeta) \sum \limits _{\xi\ge \zeta}v^q(\xi) g_1^q(\xi)\Bigr)^{1/q} \stackrel{\eqref{v_w_cond}}{\underset{a, \, b, \, \varepsilon}{\lesssim}}
$$
$$
\lesssim \Bigl( \sum \limits _{\zeta \in {\bf V}({\cal T})}g^q(\zeta) f^q(\zeta) v^q(\zeta)\Bigr)^{1/q} \le \sup _{\zeta \in {\bf V}({\cal T})} g(\zeta) v(\zeta) \Bigl( \sum \limits _{\zeta \in {\bf V}({\cal T})}f^p(\zeta)\Bigr)^{1/p}
$$
(in the case $q=1$ the value $\Bigl(\sum \limits _{\xi_0\le \eta \le \xi} g_1^{q'}(\eta)\Bigr)^{1/q'}$ is replaced by $\max _{\xi_0\le \eta \le \xi} g_1(\eta)$).

The following assertion was proved in \cite[formula (60)]{vas_besov}.

\begin{Lem}
\label{slow_gr} Let $\rho:[1, \, \infty) \rightarrow (0, \, \infty)$ be a locally absolutely continuous function, $\lim \limits _{t\to +\infty} \frac{t\rho'(t)}{\rho(t)}=0$. Then for all $\varepsilon>0$, $t\ge 1$, $y\ge 1$, we have $t^{-\varepsilon} \underset{\varepsilon, \rho}{\lesssim} \frac{\rho(ty)}{\rho(y)} \underset{\varepsilon, \rho}{\lesssim} t^\varepsilon$.
\end{Lem}

\section{The tree-like structure of the domain $\Omega \in {\cal G}_{\varphi_1, \, \dots, \, \varphi_{d-1}}$}

Let $\Omega \in {\cal G}_{\varphi_1, \, \dots, \, \varphi_{d-1}}$.

Given $1\le i\le d-1$, $k\in \N$, we define the numbers $n_{k,i}\in \N$ by the condition
\begin{align}
\label{nki_def} 2^{-n_{k,i}} \le \varphi_i(2^{-k-3}) < 2^{-n_{k,i}+1}
\end{align}
and consider the partition of the cube $(0, \, 1)^{d-1}$ into open parallelepipeds 
\begin{align}
\label{del_pr_kj}
\Delta'_{k,j} = \prod_{i=1}^{d-1} (\zeta_{k,j,i}, \, \zeta_{k,j,i} + 2^{-n_{k,i}}), \quad j\in {\bf W}_k.
\end{align}
We also set 
\begin{align}
\label{w0_def}
{\bf W}_0=\{1\}, \quad \Delta'_{0,1} = (0, \, 1)^{d-1}.
\end{align}

Since the functions $\varphi_i$ are non-decreasing, we have $n_{k-1,i} \le n_{k,i}$, $1\le i\le d-1$. Hence for each $k\ge 1$, $j\in {\bf W}_{k-1}$ the parallelepiped $\Delta'_{k-1,j}$ is divided into parallelepipeds as follows:
\begin{align}
\label{w_kj} \Delta'_{k-1,j} = \cup _{l\in {\bf W}_{k-1,j}}\Delta'_{k,l}.
\end{align}

From the third condition of \eqref{phi_prop} it follows that $2^{n_{k,i}} \underset{a_*}{\lesssim} 2^{n_{k-1,i}}$; therefore,
\begin{align}
\label{w_k1_j_1} \# {\bf W}_{k-1,j} \underset{a_*,d}{\lesssim} 1.
\end{align}

We set
\begin{align}
\label{c_kl_pm}
\begin{array}{c}
c^-_{k,l} =\inf \{\psi(x')-2^{-k}:\; x'\in \Delta'_{k-1,j}\}, \; l\in {\bf W}_{k-1,j}, \\ c^+_{k,l} = \inf \{\psi(x')-2^{-k-1}:\, x'\in \Delta'_{k,l}\}, 
\end{array}
\end{align}
\begin{align}
\label{del_kl_def} \Delta_{k,l} = \Delta'_{k,l} \times (c^-_{k,l}, \, c^+_{k,l}).
\end{align}
We show that
\begin{align}
\label{c_pm} c^+_{k,l} - c^-_{k,l} \asymp 2^{-k}.
\end{align}
Indeed, 
\begin{align}
\label{cpm_xpm}
c^-_{k,l} = \psi(x'_-) - 2^{-k}, \quad c^+_{k,l}= \psi(x'_+) - 2^{-k-1},
\end{align}
where $x'_-\in \overline{\Delta}'_{k-1,j}$, $x'_+\in \overline{\Delta}'_{k,l}$ (the bar denotes closure). Then for each $i=1, \, \dots, \, d-1$ we have $|x'_{+,i} - x'_{-,i}| \stackrel{\eqref{del_pr_kj}}{\le} 2^{-n_{k-1,i}} \stackrel{\eqref{nki_def}}{\le} \varphi_i(2^{-k-2})$. Hence, $|\psi(x'_+) - \psi(x'_-)| \stackrel{\eqref{psi_prop}}{\le} 2^{-k-2}$. This together with \eqref{cpm_xpm} implies \eqref{c_pm}.

Notice that $c_{k,l}^- = c^+_{k-1,j}$ for $k\ge 2$, $l\in {\bf W}_{k-1,j}$. We set 
\begin{align}
\label{1232}
c^+_{0,1} := c^-_{1,j} \stackrel{\eqref{w0_def},\eqref{c_kl_pm}}{=} \inf \{\psi(x')-2^{-1}:\; x'\in (0, \, 1)^{d-1}\} \stackrel{\eqref{psi_prop}}{\in} \Bigl[\frac 12, \, \frac 32\Bigr], \quad j\in {\bf W}_1.
\end{align}

Denote $\Delta_{0,1} = (0, \, 1)^{d-1} \times (0, \, c^+_{0,1})$, $$\hat \Theta = \{\Delta_{k,j}\}_{k\in \Z_+, \, j\in {\bf W}_k}.$$

We show that $\hat \Theta$ is a partition of $\Omega$. Indeed, from the construction it follows that the parallelepipeds $\Delta_{k,j}$ do not intersect pairwise. We claim that $\{\overline{\Delta}_{k,j}\}_{k\in \Z_+, \, j\in {\bf W}_k}$ is a covering of $\Omega$. Indeed, let $x=(x', \, x_d) \in \Omega$. If $x_d< c^+_{0,1}$, then $x\in \overline{\Delta}_{0,1}$. Let $x_d\ge c^+_{0,1}$. If 
\begin{align}
\label{xd_l_psi_21}
x_d<\psi(x')-2^{-1}, 
\end{align}
we set $k=0$; if $x_d\ge \psi(x')-2^{-1}$, we choose $k\in \N$ such that 
\begin{align}
\label{k_chc} x_d\in [\psi(x')-2^{-k}, \, \psi(x')-2^{-k-1}).
\end{align}
Now we choose $l\in {\bf W}_k$ such that $x'\in \overline{\Delta}'_{k,l}$. By \eqref{c_kl_pm} and \eqref{k_chc}, we have $x_d\ge c^-_{k,l}$ in the case $k\in \N$. If $x_d\le c^+_{k,l}$, then $x\in \overline{\Delta}_{k,l}$. Let $x_d> c^+_{k,l}$. We choose $t\in {\bf W}_{k+1}$ such that $x'\in \overline{\Delta}'_{k+1,t}$. By \eqref{c_kl_pm}, we have $x_d> c_{k+1,t}^-$. If $x_d\le c^+_{k+1,t}$, then $x\in \overline{\Delta} _{k+1,t}$. Let $x_d> c^+_{k+1,t}$. There is a point $x'_+\in \overline{\Delta}'_{k+1,t}$ such that $c^+_{k+1,t} = \psi(x'_+) - 2^{-k-2}$. Hence
$$
x_d> \psi(x'_+) - 2^{-k-2}\ge \psi(x') - 2^{-k-2} - |\psi(x') - \psi(x'_+)|
$$
$$
\stackrel{\eqref{psi_prop}, \eqref{nki_def}, \eqref{del_pr_kj}}{\ge} \psi(x') - 2^{-k-2} - 2^{-k-3} > \psi(x') - 2^{-k-1},
$$
which contradicts  \eqref{xd_l_psi_21} with $k=0$, and  \eqref{k_chc} with $k\in \N$.

Let ${\cal A}$ be a tree with a vertex set $\{\xi_{k,j}\}_{k\in \Z_+, \, j\in {\bf W}_k}$. As a minimal vertex, we take  $\xi_{0,1}$. The edges and the partial order are defined as follows: ${\bf V}_1(\xi_{0,1}) =\{\xi_{1,j}\}_{j\in {\bf W}_1}$; if $k\ge 2$, $j\in {\bf W}_{k-1}$, then ${\bf V}_1(\xi_{k-1,j}) = \{\xi_{k,l}:\; l\in {\bf W}_{k-1,j}\}$. From \eqref{w_k1_j_1} it follows that 
\begin{align}
\label{card_v1} {\rm card}\, {\bf V}_1(\xi) \underset{a_*,d}{\lesssim} 1.
\end{align}

The mapping $\hat F: {\bf V}({\cal A}) \rightarrow \hat \Theta$ is defined by the formula
\begin{align}
\label{hat_f_def}
\hat F(\xi_{k,j}) = \Delta _{k,j}, \; k\in \Z_+, \; j\in {\bf W}_k.
\end{align}

Let $\lambda_{k,i}$ be the length of $i$th edge of the parallelepiped $\Delta_{k,j}$. Then
\begin{align}
\label{lam_ki_est} \begin{array}{c} \lambda_{k,i} = 2^{-n_{k,i}} \stackrel{\eqref{phi_prop}, \eqref{nki_def}}{\underset{a_*}{\asymp}} \varphi_i(2^{-k}) \stackrel{\eqref{phi_prop}}{\underset{a_*}{\lesssim}} 2^{-k}, \quad i=1, \, \dots, \, d-1, \\ \lambda_{k,d} \stackrel{\eqref{del_kl_def}, \eqref{c_pm}, \eqref{1232}}{\asymp} 2^{-k}.\end{array}
\end{align}
It implies that
\begin{align}
\label{fkj_vol} |\hat F(\xi_{k,j})| \underset{a_*,d}{\asymp} 2^{-k} \prod _{i=1}^{d-1} \varphi_i(2^{-k}).
\end{align}

\section{Proof of Theorems \ref{main_entr}, \ref{main_widths}}

It suffices to consider $\Omega \in {\cal G}_{\varphi_1, \, \dots, \, \varphi_{d-1}}$.

Let ${\cal A}$ be the tree constructed in the previous section, let ${\cal A}' \subset {\cal A}$ be its subtree, and let $\xi_{k_0,j_0}$ be its minimal vertex. We denote
$$
\Omega_{{\cal A}'} = \Delta_{k_0,j_0} \cup \Bigl( \cup _{\xi _{k,j}\in {\bf V}({\cal A}'), \, \xi_{k,j} > \xi _{k_0,j_0}} \Delta'_{k,j} \times [c^-_{k,j}, \, c^+_{k,j})\Bigr),
$$
$$
\Omega^*_{{\cal A}'} = \cup _{\xi \in {\bf V}({\cal A}')} \hat F(\xi) \stackrel{\eqref{del_kl_def}, \eqref{hat_f_def}}{=} \cup _{\xi_{k,j}\in {\bf V}({\cal A}')} \Delta'_{k,j}\times (c^-_{k,j}, \, c^+_{k,j}).
$$
Then the set $\Omega_{{\cal A}'} \backslash \Omega^*_{{\cal A}'}$ has Lebesgue measure zero.

We obtain the upper estimate of the best approximation in $L_q$ of a function $f\in {\cal W}^r_p(\Omega)$ by a polynomial of degree at most $r-1$ on a subdomain $\Omega_{{\cal A}'}$. The conditions on the functions $\varphi_i$ will be more general than in Theorems \ref{main_entr}, \ref{main_widths}.

Given $0\le \alpha\le \beta\le 1$, we denote $$\Delta _{k,j} ^{\alpha,\beta} = \Delta'_{k,j} \times ((1-\alpha)c^-_{k,j} + \alpha c^+_{k,j}, \, (1-\beta)c^-_{k,j} + \beta c^+_{k,j}).$$

We set $$\mathfrak{Z} = (p, \, q, \, r, \, d, \, \varphi_1, \, \dots, \, \varphi_{d-1}).$$
\begin{Lem}
\label{emb_est1}
Suppose that
\begin{align}
\label{rdpdq_g0}
r + \frac dq - \frac dp>0
\end{align}
and for sufficiently large numbers $k$ the sequence 
\begin{align}
\label{est_seq}
k\mapsto 2^{-k(r+1/q-1/p)} \Bigl(\prod _{i=1}^{d-1} \varphi_i(2^{-k})\Bigr)^{1/q-1/p}
\end{align}
is non-increasing. In addition, suppose that, for $q=\infty$, there are numbers $\hat c_1>0$, $\hat c_2>0$ such that
\begin{align}
\label{geom_pr_infty} \frac{2^{-(k+j)(r-1/p)} \Bigl(\prod _{i=1}^{d-1} \varphi_i(2^{-k-j})\Bigr)^{-1/p}}{2^{-k(r-1/p)} \Bigl(\prod _{i=1}^{d-1} \varphi_i(2^{-k})\Bigr)^{-1/p}} \le \hat c_1\cdot 2^{-\hat c_2j}, \quad j\in \Z_+.
\end{align}
Then, for any function $f\in C^\infty(\Omega)$ such that 
\begin{align}
\label{f_dk0j0_0}
f|_{\Delta_{k_0,j_0}^{0,1/2}} = 0,
\end{align}
we have
\begin{align}
\label{nul_est} \|f\|_{L_q(\Omega_{{\cal A}'})} \underset{\mathfrak{Z}}{\lesssim} 2^{-k_0(r + 1/q - 1/p)} \Bigl(\prod _{i=1}^{d-1} \varphi_i(2^{-k_0})\Bigr)^{1/q-1/p} \|\nabla^r f\| _{L_p(\Omega_{{\cal A}'})}.
\end{align}
\end{Lem}

\begin{proof} We denote by $x'_{k,j}$ the center of the parallelepiped $\Delta'_{k,j}$. We set 
\begin{align}
\label{x0_def_d12}
x_0 := \Bigl(x'_{k_0,j_0}, \, \frac{3c^-_{k_0,j_0} + c^+_{k_0,j_0}}{4}\Bigr) \in \Delta_{k_0,j_0}^{0,1/2}.
\end{align}

Let $x\in \Omega_{{\cal A}'} \backslash \Delta _{k_0,j_0}^{0,1/2}$. We construct the integral representation for $f(x)$ in terms of $\nabla^r f$.

First we consider the case $r = 1$.

We construct the polygonal line connecting $x$ and $x_0$. 

Let $s\in \Z_+$, $j_s \in {\bf W}_{k_0+s}$ be such that $x\in \Delta_{k_0+s,j_s}$, and let $(\xi_{k_0,j_0}, \, \xi_{k_0+1,j_1}, \, \dots, \, \xi_{k_0+s,j_s})$ be the chain in ${\cal A}'$ connecting $\xi_{k_0,j_0}$ and $\xi_{k_0+s,j_s}$.

If $s=0$, we connect $x$ and $x_0$ by the segment.

Let $s\in \N$. For $1\le l\le s-1$ we set $x_{(l)} = (x'_{k_0+l,j_l}, \, c^-_{k_0+l, j_l})$. We also denote $$\tilde x_{(s-1)} = \Bigl( x'_{k_0+s,j_s}, \, \frac{c^-_{k_0+s-1, j_{s-1}} + c^+_{k_0+s-1, j_{s-1}}}{2}\Bigr).$$
We connect the point $x$ with $\tilde x_{(s-1)}$ by the segment. If $s=1$, we connect $\tilde x_{(s-1)}$ with $x_0$ by the segment. If $s\ge 2$, then we connect $\tilde x_{(s-1)}$ with $x_{(s-1)}$ by the segment, then for each $l\in \{2, \, \dots, \, s-1\}$ we connect $x_{(l)}$ and $x_{(l-1)}$, $x_{(1)}$ and $x_0$ by the segment.

We take the natural parametrization on the constructed polygonal line with respect to the Euclidean norm $|\cdot|$ on $\R^d$ and obtain the piecewise-affine function $\gamma_{x,x_0}:[0, \, T_{x,x_0}] \rightarrow \Omega _{{\cal A}'}$ such that 
\begin{align}
\label{beg_end}
\gamma_{x,x_0}(0)=x, \quad \gamma_{x,x_0}(T_{x,x_0}) = x_0.
\end{align}

We define the numbers $t_{(l)}$ $(1\le l\le s-1)$ by the equation $\gamma_{x,x_0}(t_{(l)}) = x_{(l)}$, and the number $\tilde t_{(s-1)}$, by the equation $\gamma_{x,x_0}(\tilde t_{(s-1)}) = \tilde x_{(s-1)}$.

Notice that by \eqref{lam_ki_est}, for $s\ge 2$,
\begin{align}
\label{tltl1_dif}
\begin{array}{c}
T_{x,x_0}-t_{(1)} = |x_0-x_{(1)}| \underset{d, a_*}{\asymp} 2^{-k_0}, \quad t_{(l)}- t_{(l+1)} = |x_{(l)} - x_{(l+1)}| \underset{d, a_*}{\asymp} 2^{-k_0-l}, \; 1\le l\le s-2, \\ t_{(s-1)} - \tilde t_{(s-1)} = |x_{(s-1)} - \tilde x_{(s-1)}| \underset{d, a_*}{\asymp} 2^{-k_0-s}, \; \tilde t_{(s-1)} =|x-\tilde x_{(s-1)}| \underset{d, a_*}{\asymp} 2^{-k_0-s},
\end{array}
\end{align}
for $s=1$,
\begin{align}
\label{tltl1_dif1}
T_{x,x_0}-\tilde t_{(1)} \underset{d, a_*}{\asymp} 2^{-k_0}, \quad \tilde t_{(1)} \underset{d, a_*}{\asymp} 2^{-k_0},
\end{align}
and for $s=0$,
\begin{align}
\label{tltl1_dif2}
T_{x,x_0} \underset{d, a_*}{\asymp} 2^{-k_0}.
\end{align}

We define the numbers $r_i(t)$, $t\in [0, \, T_{x,x_0}]$, $i=1, \, \dots, \, d$, as follows: for $1\le i\le d-1$ we set
\begin{align}
\label{gam_z_knot} r_i(t) = \begin{cases} 0, & t = 0, \\ 2^{-n_{k_0+s-1},i}, & t= \tilde t_{(s-1)}, \\  2^{-n_{k_0+l},i}, & t=t_{(l)}, \; 1\le l\le s-1, \\ 2^{-n_{k_0,i}}, & t=T_{x,x_0}, \end{cases}
\end{align}
for $i=d$,
\begin{align}
\label{gam_z_knot_d} r_d(t) = \begin{cases} 0, & t = 0, \\ 2^{-k_0-s+1}, & t= \tilde t_{(s-1)}, \\ 2^{-k_0-l}, & t=t_{(l)}, \; 1\le l\le s-1, \\  2^{-k_0}, & t=T_{x,x_0};\end{cases}
\end{align}
then we interpolate $r_i(t)$ $(1\le i\le d)$ as the piecewise-affine function on $[0, \, T_{x,x_0}]$ with knots at $0, \, \tilde t_{(s-1)}, \, t_{(s-1)}, \, t_{(s-2)}, \,   \dots, \, t_{(1)}, \, T_{x,x_0}$.

From \eqref{phi_prop}, \eqref{nki_def} it follows that $2^{-n_{k,i}} \underset{a_*}{\lesssim} 2^{-k}$. Hence, if $s\ge 2$, for $1\le i\le d$ we have $|r_i(t_{(1)}) - r_i(T_{x,x_0})| \lesssim 2^{-k_0}$, $|r_i(t_{(l+1)}) - r_i(t_{(l)})| \lesssim 2^{-k_0-l}$, $1\le l\le s-2$, $|r_i(\tilde t_{(s-1)}) - r_i(t_{(s-1)})| \lesssim 2^{-k_0-s}$, $|r_i(0)-r_i(\tilde t_{(s-1)})| \lesssim 2^{-k_0-s}$. This together with \eqref{tltl1_dif} implies that
\begin{align}
\label{dot_ri_t}
|\dot r_i(t)| \underset{d,a_*}{\lesssim} 1.
\end{align}
For $s=0$ and $s=1$ the estimate \eqref{dot_ri_t} is also true (see \eqref{tltl1_dif1}, \eqref{tltl1_dif2}).

Let $0<c\le 1$.
Denote 
\begin{align}
\label{bxx0t_def}
B_{x,x_0,t} = \gamma_{x,x_0}(t) + \prod _{i=1}^d [ - c\cdot r_i(t), \, c\cdot r_i(t)], \; G_x = \cup _{t\in [0, \, T_{x,x_0}]} B_{x,x_0,t}.
\end{align}

Given $(z, \, t) = (z_1, \, \dots, \, z_d, \, t) \in [-1, \, 1]^d\times [0, \, T_{x,x_0}]$, we set
\begin{align}
\label{gam_xx0z_def}
\gamma_{x,x_0,z}(t) = \gamma_{x,x_0}(t) + c\cdot (r_1(t)z_1, \, \dots, \, r_d(t)z_d) \in B_{x,x_0,t}.
\end{align}

Denote by $[\xi_{k_0,j_0}, \, \xi_{k_0+s,j_s}]$ the subtree in ${\cal A}$, whose vertex set is the chain $\xi_{k_0,j_0}, \, \dots$, $\xi_{k_0+s,j_s}$.

We choose $c = c(d, \, a_*)\in (0, \, 1]$ so that the following inclusions and inequalities hold:
\begin{align}
\label{incl_c} 
\begin{array}{c}
\gamma_{x,x_0,z}(t)\in \Omega_{[\xi_{k_0,j_0}, \, \xi_{k_0+s,j_s}]}, \quad t \in [0, \, T_{x,x_0}], \; x\in \Delta_{k_0+s,j_s}\backslash \Delta_{k_0,j_0}^{0,1/2}, \; z\in [-1, \, 1]^d; \\ \gamma_{x,x_0,z}(T_{x,x_0}) \in  \Delta_{k_0,j_0}^{0,1/2},
\end{array}
\end{align}
\begin{align}
\label{dot_gamma} \frac 12 \le |\dot \gamma_{x,x_0,z}(t)| \le 2
\end{align}
(\eqref{dot_gamma} is possible by  \eqref{dot_ri_t} and the choice of the natural parametrization of the polygonal line  $\gamma_{x,x_0}$; the second inclusion in \eqref{incl_c} follows from \eqref{x0_def_d12} and \eqref{beg_end}). From \eqref{bxx0t_def}, \eqref{gam_xx0z_def} and the first inclusion \eqref{incl_c} it follows that
\begin{align}
\label{gx_cep} G_x \subset \Omega _{[\xi_{k_0,j_0}, \, \xi_{k_0+s,j_s}]}.
\end{align}

From \eqref{beg_end}, \eqref{gam_z_knot}, \eqref{gam_z_knot_d}, \eqref{gam_xx0z_def} it follows that $\gamma _{x,x_0,z}(0)=x$ for each $z\in [-1, \, 1]^d$. Applying the Newton--Leibnitz formula and taking into account \eqref{f_dk0j0_0} and the second inclusion in \eqref{incl_c}, we get for each $z\in [-1, \, 1]^d$
$$
f(x) = f(\gamma _{x,x_0,z}(0)) - f(\gamma_{x,x_0,z}(T_{x,x_0})) = -\int _0^{T_{x,x_0}} \langle \nabla f(\gamma_{x,x_0,z}(t)), \, \dot \gamma_{x,x_0,z}(t)\rangle \, dt.
$$
We integrate this equality over $z\in [-1, \, 1]^d$:
$$
2^d f(x) = -\int \limits _{[-1, \, 1]^d} \int _0^{T_{x,x_0}} \langle \nabla f(\gamma_{x,x_0,z}(t)), \, \dot \gamma_{x,x_0,z}(t)\rangle \, dt\, dz =:I.
$$
We change the variables to  $(t, \, z) \mapsto (t, \, y)$ from the condition  $y = \gamma_{x,x_0,z}(t)$, and set 
\begin{align}
\label{wty_def}
w(t, \, y) = \dot \gamma_{x,x_0,z}(t)\cdot c^{-d}\prod _{i=1}^d (r_i(t))^{-1}.
\end{align}
Taking into account \eqref{bxx0t_def}, we get
\begin{align}
\label{2dfx_int}
2^d f(x) = -\int \limits _0^{T_{x,x_0}} \int \limits _{B_{x,x_0,t}}\langle \nabla f(y), \, w(t, \, y) \rangle \, dy\, dt = \int \limits _{G_x} \langle \nabla f(y), \, K(x, \, y)\rangle \, dy,
\end{align}
where 
\begin{align}
\label{kxy_w}
K(x, \, y) = -\int \limits _{t:\; y\in B_{x,x_0,t}} w(y, \, t)\, dt.
\end{align}

Let $y\in B_{x,x_0,t}$. We set $t_{(0)} = T_{x,x_0}$. Then, if $c >0$ is sufficiently small (the maximal value depends only on $a_*$ and $d$), the following inclusions hold. If $y\in \Delta_{k_0+l,j_l}$, $0\le l\le s-2$, then 
\begin{align}
\label{t_y_bxx0}
t \in \begin{cases} [t_{(2)}, \, t_{(0)}], & l=0\le s-3, \\ [t_{(l+2)}, \, t_{(l-1)}], & 1\le l\le s-3, \\ [\tilde t_{(s-1)}, \, t_{(s-3)}], & l=s-2\ge 1, \\ [\tilde t_{(1)}, \, t_{(0)}], & l = s-2=0;\end{cases}
\end{align}
this together with \eqref{phi_prop}, \eqref{nki_def}, \eqref{gam_z_knot}, \eqref{gam_z_knot_d} implies
\begin{align}
\label{r_i_est1} r_i(t) \underset{\mathfrak{Z}}{\asymp} \begin{cases} \varphi_i(2^{-k_0-l}), & 1\le i\le d-1, \\ 2^{-k_0-l}, & i=d.\end{cases}
\end{align}
Let $y\in \Delta_{k_0+s-1,j_{s-1}} \cup \Delta _{k_0+s, j_s}$. Then, if $c=c(d, \, a_*)$ is sufficiently small, we have 
\begin{align}
\label{t_por_xmy}
t\underset{\mathfrak{Z}}{\asymp} |x_d-y_d|,
\end{align}
\begin{align}
\label{r_i_est2}
r_i(t) \underset{\mathfrak{Z}}{\asymp} \begin{cases} |x_d-y_d|\cdot 2^{k_0+s}\varphi_i(2^{-k_0-s}), & 1\le i\le d-1, \\ |x_d-y_d|, & i=d.\end{cases}
\end{align}
Further, if $c=c(d, \, a_*)$ is sufficiently small, by \eqref{r_i_est1}, \eqref{r_i_est2} and the definition of $\gamma_{x,x_0}$ we get
\begin{align}
\label{max_xi} \max_{1\le i\le d-1} 2^{n_{k_0+s,i}} |x_i -y_i| \underset{\mathfrak{Z}}{\lesssim} 2^{k_0+s} |x_d - y_d|.
\end{align}

By \eqref{tltl1_dif}, \eqref{t_y_bxx0}, \eqref{t_por_xmy}, we have
$$
{\rm mes}\,\{t:\; y\in B_{x,x_0,t}\} \underset{\mathfrak{Z}}{\lesssim} \begin{cases} 2^{-k_0-l}, & y\in \Delta_{k_0+l,j_l}, \; l\le s-2, \\ |x_d - y_d|, & y\in \Delta _{k_0+s-1,j_{s-1}} \cup \Delta_{k_0+s,j_s};\end{cases}
$$
taking into account \eqref{dot_gamma}, \eqref{wty_def}, \eqref{kxy_w}, \eqref{r_i_est1}, \eqref{r_i_est2}, we get
\begin{align}
\label{kxy_est} |K(x, \, y)| \underset{\mathfrak{Z}}{\lesssim} 
\begin{cases} 
\prod _{i=1}^{d-1} (\varphi _i(2^{-k_0-l}))^{-1}, & y\in \Delta_{k_0+l,j_l}, \; l\le s-2, \\ |x_d-y_d|^{1-d} 2^{-(k_0+s)(d-1)} \prod _{i=1}^{d-1}(\varphi _i(2^{-k_0-s}))^{-1}, & y\in \Delta _{k_0+s-1,j_{s-1}} \cup \Delta_{k_0+s,j_s}.
\end{cases}
\end{align}

Thus, we have got for $r=1$ the integral representation \eqref{2dfx_int} with the kernel $K(x, \, y)$ satisfying \eqref{kxy_est}. For $r\ge 2$, we use \eqref{2dfx_int}, apply Reshetnyak's method \cite{resh2} and get
\begin{align}
\label{r_gr}
|f(x)| \underset{\mathfrak{Z}}{\lesssim} \int \limits _{G_x} |\nabla^r f(y)|\cdot |x-y|^{r-1} \cdot |K(x, \, y)| \, dy.
\end{align}

Denote
\begin{align}
\label{til_k_def}
\tilde K(x, \, y) = |x-y|^{r-1} \cdot |K(x, \, y)|.
\end{align}

Now we prove \eqref{nul_est}. 

First we consider the case $q<\infty$.
We have
$$
\Bigl(\int \limits _{\Omega_{{\cal A}'}} |f(x)|^q\, dx\Bigr)^{1/q} = \Bigl(\sum \limits _{\xi \in {\bf V}({\cal A}')} \int \limits _{\hat F(\xi)} |f(x)|^q\, dx \Bigr)^{1/q} \stackrel{\eqref{r_gr}, \eqref{til_k_def}}{\underset{\mathfrak{Z}}{\lesssim}}
$$
$$
\lesssim \Bigl(\sum \limits _{\xi \in {\bf V}({\cal A}')} \int \limits _{\hat F(\xi)} \Bigl( \int \limits _{G_x} |\nabla^r f(y)|\cdot \tilde K(x, \, y)\, dy\Bigr)^q\, dx\Bigr)^{1/q} \stackrel{\eqref{gx_cep}}{\le}
$$
$$
\le \Bigl(\sum \limits _{\xi \in {\bf V}({\cal A}')} \int \limits _{\hat F(\xi)} \Bigl( \sum \limits _{\xi_{k_0,j_0}\le \eta \le \xi,\, \rho(\xi, \, \eta) \ge 2} \int \limits _{G_x\cap \hat F(\eta)} |\nabla^r f(y)|\cdot \tilde K(x, \, y)\, dy\Bigr)^q\, dx\Bigr)^{1/q} +
$$
$$
+ \Bigl(\sum \limits _{\xi \in {\bf V}({\cal A}')} \int \limits _{\hat F(\xi)} \Bigl| \sum \limits _{\xi_{k_0,j_0}\le \eta\le \xi,\, \rho(\xi, \, \eta) \le 1} \int \limits _{G_x\cap \hat F(\eta)} |\nabla^r f(y)|\cdot \tilde K(x, \, y)\, dy\Bigr|^q\, dx\Bigr)^{1/q} =: I_1+I_2.
$$

We estimate $I_1$. Let $\xi = \xi_{k_0+s,j_s}$, $\eta = \xi_{k_0+l,j_l}$, $0\le l\le s-2$, $x\in \hat F(\xi)$, $y\in \hat F(\eta)$. Then $|x-y| \stackrel{\eqref{lam_ki_est}}{\underset{\mathfrak{Z}}{\asymp}} 2^{-k_0-l}$,
\begin{align}
\label{i1_est1}
\begin{array}{c}
\int \limits _{G_x\cap \hat F(\eta)} |\nabla^r f(y)| \cdot \tilde K(x, \, y)\, dy \stackrel{\eqref{kxy_est}, \eqref{til_k_def}}{\underset{\mathfrak{Z}}{\lesssim}} \prod _{i=1} ^{d-1} (\varphi_i(2^{-k_0-l})) ^{-1} 2^{-(k_0+l)(r-1)}|\hat F(\eta)| ^{1/p'}  \|\nabla^r f\| _{L_p(\hat F(\eta))} \stackrel{\eqref{fkj_vol}}{\underset{\mathfrak{Z}}{\asymp}}
\\
\asymp \prod _{i=1} ^{d-1} (\varphi_i(2^{-k_0-l})) ^{-1/p} 2^{-(k_0+l)(r-1/p)} \|\nabla^r f\| _{L_p(\hat F(\eta))}.
\end{array}
\end{align}
Given $t\in {\bf W}_{k_0+s}$, we denote by $j_{l,t}$ the index from ${\bf W}_{k_0+l}$ such that $\xi_{k_0+l,j_{l,t}} \le \xi_{k_0+s,t}$. We also put $\tilde {\bf W}_{k_0+s} = {\bf W}_{k_0+s} \cap {\bf V}({\cal A}')$. From \eqref{i1_est1} we obtain 
$$
I_1^q \underset{\mathfrak{Z}}{\lesssim} \sum \limits _{s\ge 2, \, t\in \tilde{\bf W}_{k_0+s}}\int \limits _{\Delta _{k_0+s,t}}\Bigl( \sum \limits _{l=0}^{s-2} 2^{-(k_0+l)(r-1/p)}\prod _{i=1} ^{d-1} (\varphi_i(2^{-k_0-l})) ^{-1/p}  \|\nabla^r f\| _{L_p(\Delta_{k_0+l,j_{l,t}})}\Bigr)^q\, dx \stackrel{\eqref{fkj_vol}}{\underset{\mathfrak{Z}}{\asymp}}
$$
$$
\asymp \sum \limits _{s\ge 2, \, t\in \tilde{\bf W}_{k_0+s}}2^{-k_0-s}\prod _{i=1}^{d-1} \varphi_i(2^{-k_0-s}) \Bigl( \sum \limits _{l=0}^{s-2} 2^{-(k_0+l)(r-1/p)}\prod _{i=1} ^{d-1} (\varphi_i(2^{-k_0-l})) ^{-1/p}  \|\nabla^r f\| _{L_p(\Delta_{k_0+l,j_{l,t}})}\Bigr)^q.
$$
Hence
\begin{align}
\label{i1_est_s}
I_1 \underset{\mathfrak{Z}}{\lesssim} \mathfrak{S}^{p,q}_{{\cal A}',g,v} \|\nabla^r f\|_{L_p(\Omega_{{\cal A}'})},
\end{align}
where
$$
g(\xi) = 2^{-(k_0+l)(r-1/p)}\prod _{i=1} ^{d-1} (\varphi_i(2^{-k_0-l})) ^{-1/p}, \quad v(\xi) = 2^{-(k_0+l)/q}\prod _{i=1}^{d-1} \varphi_i^{1/q}(2^{-k_0-l})
$$
for $\xi \in {\bf V}_l(\xi_{k_0,j_0})$.

We check the condition \eqref{v_w_cond}. Indeed, let $\xi \in {\bf V}_l(\xi_{k_0,j_0})\cap {\bf V}({\cal A}')$, $j\in \Z_+$. Then
$$
\sum \limits _{\eta \in {\bf V}_j(\xi)\cap {\bf V}({\cal A}')} v^q(\eta) \le 2^{-k_0-l-j} \prod _{i=1}^{d-1} \varphi_i(2^{-k_0-l-j}) \prod _{i=1} ^{d-1} \frac{2^{n_{k_0+l+j,i}}}{2^{n_{k_0+l,i}}} \stackrel{\eqref{phi_prop}, \eqref{nki_def}}{\underset{a_*,d}{\asymp}}
$$
$$
\asymp 2^{-j}\cdot 2^{-k_0-l} \prod _{i=1}^{d-1} \varphi_i(2^{-k_0-l}) = 2^{-j} v^q(\xi).
$$

Thus, \eqref{sum_est_0} holds. Taking into account that the sequence \eqref{est_seq} is non-increasing by the conditions of the present lemma, we get
$$\mathfrak{S}^{p,q}_{{\cal A}',g,v}\underset{\mathfrak{Z}}{\lesssim} 2^{-k_0(r-1/p+1/q)}\prod _{i=1} ^{d-1} (\varphi_i(2^{-k_0})) ^{1/q-1/p},$$
\begin{align}
\label{i1_est} I_1 \stackrel{\eqref{i1_est_s}}{\underset{\mathfrak{Z}}{\lesssim}} 2^{-k_0(r+1/q-1/p)}\prod _{i=1} ^{d-1} (\varphi_i(2^{-k_0})) ^{1/q-1/p}  \|\nabla^r f\|_{L_p(\Omega_{{\cal A}'})}.
\end{align}

Now we estimate $I_2$.

\begin{Not}
\label{notation}
Given $\xi \in {\bf V}({\cal A})$, we denote by $T_\xi:\R^d \rightarrow \R^d$ the affine mapping, which is a composition of dilations along coordinate axes and a translation, which 
maps  $\hat F(\xi)$ to $[0,\, 1]^d$.
\end{Not}

Let $\xi = \xi_{k_0+s,j_s}$. We set 
$$
E_\xi = \begin{cases}\hat F(\xi_{k_0+s,j_s}) \sqcup \hat F(\xi_{k_0+s-1,j_{s-1}}), & s\ge 1, \\ \hat F(\xi_{k_0+s,j_s}), & s=0,\end{cases}
$$
$$C_\xi = \prod _{i=1} ^{d-1} (\varphi_i(2^{-k_0-s}))^{-1} \cdot 2^{-(k_0+s)(d-1)},$$ $M_\xi = T_\xi(E_\xi)$. Then (see \eqref{card_v1}, \eqref{kxy_est}, \eqref{til_k_def} and the condition $p\le q$)
\begin{align}
\label{i2_set}
I_2 \underset{\mathfrak{Z}}{\lesssim} \max _{\xi \in {\bf V({\cal A}')}} A_{2,\xi} \|\nabla^r f\|_{L_p(\Omega_{{\cal A}'})},
\end{align}
where $A_{2,\xi}$ is the minimal constant $D$ in the inequality
$$
C_\xi\Bigl(\int \limits_{\hat F(\xi)}\Bigl(\int \limits _{E_\xi \cap G_x} |\varphi(y)|\cdot |x-y|^{r-1}|x_d-y_d|^{1-d}\, dy\Bigr)^q\, dx\Bigr)^{1/q} \le D \|\varphi\| _{L_p(E_\xi)}.
$$
Changing the variables via the mapping  $T_\xi$ and taking into account \eqref{lam_ki_est}, \eqref{fkj_vol}, \eqref{max_xi}, we get $A_{2,\xi}\underset{\mathfrak{Z}}{\lesssim}\tilde D$, where $\tilde D$ is the minimal constant in the inequality
$$
\prod _{i=1}^{d-1} (\varphi_i(2^{-k_0-s}))^{-1}\cdot 2^{-(k_0+s)(d-1)} \cdot 2^{-(k_0+s)(1/q+1)} \prod _{i=1}^{d-1} (\varphi_i(2^{-k_0-s}))^{1/q+1} \cdot 2^{-(k_0+s)(r-d)}   \times$$$$\times\Bigl(\int \limits _{[0, \, 1]^d} \Bigl(\int \limits _{M_\xi} |\tilde \varphi(z)|\cdot|w-z|^{r-d}\, dz\Bigr)^q\, dw\Bigr)^{1/q}\le $$$$\le\tilde D \cdot 2^{-(k_0+s)/p} \prod _{i=1}^{d-1} (\varphi_i(2^{-k_0-s}))^{1/p}\|\tilde \varphi\| _{L_p(M_\xi)}.
$$
Recall that \eqref{rdpdq_g0} holds. Applying the theorem about the continuity of the convolution operator with the kernel $|x-y|^{r-d}$ (see \cite[Ch. I, \S 6]{sob_book}, \cite[section 1.4.1, Theorem 2]{mazya_book}), we get
$$
A_{2,\xi} \underset{\mathfrak{Z}}{\lesssim} 2^{-(k_0+s)(r + 1/q-1/p)}\prod _{i=1}^{d-1} (\varphi_i(2^{-k_0-s}))^{1/q-1/p}.
$$
Since the sequence \eqref{est_seq} is non-increasing, this together with \eqref{i2_set} implies
\begin{align}
\label{i2_est}
I_2 \underset{\mathfrak{Z}}{\lesssim} 2^{-k_0(r + 1/q-1/p)}\prod _{i=1}^{d-1} (\varphi_i(2^{-k_0}))^{1/q-1/p} \|\nabla^r f\| _{L_p(\Omega _{{\cal A}'})}.
\end{align}
From \eqref{i1_est} and \eqref{i2_est} we get \eqref{nul_est}.

In the case $q=\infty$ the proof is more simple. We have
$$
\|f\|_{L_\infty(\Omega_{{\cal A}'})} \underset{\mathfrak{Z}}{\lesssim} \sup _{\xi \in {\bf V}({\cal A}')}\sum _{\xi_{k_0,j_0}\le \eta\le \xi, \, \rho(\xi, \, \eta) \ge 2} \int \limits _{G_x \cap \hat F(\eta)} |\nabla^r f(y)| \tilde K(x, \, y) \, dy +$$$$+\sup _{\xi \in {\bf V}({\cal A}')} \sum _{\xi_{k_0,j_0}\le \eta\le \xi, \, \rho(\xi, \, \eta) \le 1} \int \limits _{G_x \cap \hat F(\eta)} |\nabla^r f(y)| \tilde K(x, \, y) \, dy.
$$
Recall the notation $\tilde{\bf W}_{k_0+s} = {\bf W}_{k_0+s} \cap {\bf V}({\cal A}')$. By \eqref{i1_est1}, the first summand can be estimated (up to a multiplicative constant depending only on $\mathfrak{Z}$) by
$$
\sup_{s\ge 2, \, t\in \tilde{\bf W}_{k_0+s}} \sum \limits _{l=0}^{s-2} 2^{-(k_0+l)(r-1/p)}\prod _{i=1} ^{d-1} (\varphi_i(2^{-k_0-l})) ^{-1/p}  \|\nabla^r f\| _{L_p(\Delta_{k_0+l,j_{l,t}})} \stackrel{\eqref {geom_pr_infty}}{\underset{\mathfrak{Z}}{\lesssim}}$$$$ \lesssim 2^{-k_0(r-1/p)}\prod _{i=1} ^{d-1} (\varphi_i(2^{-k_0})) ^{-1/p} \|\nabla^r f\| _{L_p(\Omega_{{\cal A}'})}.
$$
The second summand can be estimated as in the case $q<\infty$.
\end{proof}

\begin{Lem}
\label{lem2est}
Let the conditions of Lemma {\rm \ref{emb_est1}} hold. Then, for each vertex $\xi_{k_0,j_0}\in {\bf V}({\cal A})$,
there is a continuous linear projection $Q_{k_0,j_0}:L_q(\Omega) \rightarrow {\cal P}_{r-1}(\Omega)$ such that for any subtree $({\cal A}', \, \xi_{k_0,j_0})$ in ${\cal A}$ and any function $f\in {\cal W}^r_p(\Omega)$
\begin{align}
\label{appr_est} \|f- Q_{k_0,j_0}f\|_{L_q(\Omega _{{\cal A}'})} \underset{\mathfrak{Z}}{\lesssim} 2^{-k_0(r + 1/q-1/p)}\prod _{i=1}^{d-1} (\varphi_i(2^{-k_0}))^{1/q-1/p} \|\nabla^r f\|_{L_p(\Omega_{{\cal A}'})}.
\end{align}
\end{Lem}
\begin{proof}
It suffices to consider $f\in {\cal W}^r_p(\Omega) \cap C^\infty(\Omega)$. Indeed, let $f\in {\cal W}^r_p(\Omega)$. Then there is a sequence $\{f_m\}_{m\in \N} \in {\cal W}^r_p(\Omega) \cap C^\infty(\Omega)$ such that $\|\nabla^r f_m -\nabla^r f\|_{L_p(\Omega)} \underset{m\to \infty}{\to} 0$; if, in addition, $f\in L_q(\Omega)$, we also have $\|f_m-f\|_{L_q(\Omega)} \underset{m\to \infty}{\to} 0$. If the assertion of Lemma is proved for functions from ${\cal W}^r_p(\Omega) \cap C^\infty(\Omega)$, the sequence $\{f_m - Q_{k_0,j_0}f_m\}_{m\in \N}$ is fundamental in $L_q(\Omega_{{\cal A}'})$ and, therefore, it converges to some function $\tilde f\in L_q(\Omega_{{\cal A}'})$. Hence $\nabla^r \tilde f = \nabla^r f$ and, therefore, $\tilde f - f \in {\cal P}_{r-1}(\Omega_{{\cal A}'})$ and $f|_{\Omega_{{\cal A}'}} \in L_q(\Omega_{{\cal A}'})$. In particular, taking ${\cal A}' = {\cal A}$, we get $f\in L_q(\Omega)$. Hence $f_m-Q_{k_0,j_0}f_m \underset{m\to \infty}{\to} f - Q_{k_0,j_0}f$ in $L_q(\Omega_{{\cal A}'})$ (by the continuity of $Q_{k_0,j_0}$); it remains to pass to the limit in the order inequality.

Further we consider functions from ${\cal W}^r_p(\Omega) \cap C^\infty(\Omega )$.

First we construct the linear operator $P_{k_0,j_0}:{\cal W}^r_p(\Omega) \cap C^\infty(\Omega ) \rightarrow {\cal P}_{r-1}(\Omega)$ such that
\begin{align}
\label{appr_est1} \|f- P_{k_0,j_0}f\|_{L_q(\Omega _{{\cal A}'})} \underset{\mathfrak{Z}}{\lesssim} 2^{-k_0(r + 1/q-1/p)}\prod _{i=1}^{d-1} (\varphi_i(2^{-k_0}))^{1/q-1/p} \|\nabla^r f\|_{L_p(\Omega_{{\cal A}'})}
\end{align}
(it may be unbounded with respect to the $L_q$-norm).

Given multi-index $\beta = (\beta_1, \, \dots, \, \beta_d)$ and vector $x=(x_1, \, \dots, \, x_d)\in \R^d$, we denote $x^\beta = x_1^{\beta_1}\dots x_d^{\beta_d}$, $\beta! = \beta_1!\dots \beta_d!$.

Given $f\in C^\infty((0, \, 1)^d)\cap {\cal W}^r_p((0, \, 1)^d)$, $x\in \R^d$, we set
$$
P_{(r-1)}f(x) = 4\int \limits _{(0, \, 1)^{d-1}\times (0, \, 1/4)} \sum \limits _{|\beta| \le r-1} \partial _\beta f(y) \frac{(x-y)^\beta}{\beta!}\, dy.
$$
Then $P_{(r-1)}f\in {\cal P}_{r-1}(\R^d)$. Notice that if $\alpha$ is a multi-index of order $l\in \{0, \, \dots, \, r-1\}$, then 
\begin{align}
\label{da_pr}
\partial _\alpha P_{(r-1)}f = P_{(r-1-l)}(\partial _\alpha f).
\end{align}
In addition, for $l\in \{0, \, 1, \, \dots, \, r-1\}$, $g\in C^\infty((0, \, 1)^d) \cap {\cal W}^{r-l}_p((0, \, 1)^d)$ we have
\begin{align}
\label{fprf_lp}
\|g-P_{(r-1-l)}g\|_{L_p([0, \, 1]^{d-1} \times [1/2, \, 1])} \underset{p,r,d}{\lesssim} \|\nabla^{r-l} g\| _{L_p([0, \, 1]^d)}
\end{align}
(to this end, we construct the integral representation for $g(x)$ at points $x\in [0, \, 1]^{d-1} \times [1/2, \, 1]$ by the formula
$$
g(x) = 4\int \limits _{[0, \, 1]^{d-1}\times [0, \, 1/4]} g(y)\, dy + \int \limits _{[0, \, 1]^d} T(x, \, y) \nabla g(y)\, dy, \quad |T(x, \, y)| \underset{d}{\lesssim} |x-y|^{1-d},
$$
and in the case $l\le r-2$ we use the method from \cite{resh2}).
Hence for $f\in C^\infty((0, \, 1)^d) \cap {\cal W}^r_p((0, \, 1)^d)$ we have
\begin{align}
\label{da_nrl}
\begin{array}{c}
\|\partial_\alpha(f - P_{(r-1)}f)\|_{L_p([0, \, 1]^{d-1}\times [1/2, \, 1])} \stackrel{\eqref{da_pr}}{=} 
\|\partial _\alpha f - P_{(r-l-1)} \partial _\alpha f \| _{L_p([0, \, 1]^{d-1} \times [1/2, \, 1])} \stackrel{\eqref{fprf_lp}}{\underset{p,r,d}{\lesssim}} \\ \lesssim\|\nabla^{r-l} (\partial_\alpha f)\|_{L_p([0, \, 1]^d)}.
\end{array}
\end{align}

Now let the function $f$ be defined on $\Delta_{k_0,j_0}$, and let $T_{\xi_{k_0,j_0}}$ be the affine operator from Notation \ref{notation}. We set $h_f(x) = f(T^{-1}_{\xi_{k_0,j_0}}x)$, where $x\in (0, \, 1)^d$, $$P_{k_0,j_0}f(x) = (P_{(r-1)}h_f)(T_{\xi_{k_0,j_0}}x), \quad x\in \Omega.$$ 

Recall that $\lambda_{k_0,i}$ is the length of the $i$th edge of the parallelepiped $\Delta_{k_0,j_0}$; it satisfies \eqref{lam_ki_est} with $k=k_0$.

Let $\alpha=(\alpha_1, \, \dots, \, \alpha_d)$, $|\alpha| = l\in \{0, \, \dots, \, r-1\}$. Then
$$
\|\partial _\alpha(f-P_{k_0,j_0}f) \| _{L_p (\Delta _{k_0,j_0}^{1/2,1})} = |\Delta_{k_0,j_0}|^{1/p} \prod_{i=1}^d \lambda_{k_0,i}^{-\alpha_i} \|\partial_\alpha (h_f-P_{(r-1)}h_f)\| _{L_p([0, \, 1]^{d-1}\times [1/2,\, 1])} \stackrel{\eqref{da_nrl}}{\underset{p,r,d}{\lesssim}}
$$
$$
\lesssim |\Delta_{k_0,j_0}|^{1/p} \prod_{i=1}^d \lambda_{k_0,i}^{-\alpha_i} \|\nabla^{r-l}(\partial_\alpha h_f) \|_{L_p([0, \, 1]^d)} =: M.
$$
Let $\beta = (\beta_1, \, \dots, \, \beta_d)$ be a multi-index of order $r-l$. Then
$$
\|\partial_\beta\partial_\alpha h_f \|_{L_p([0, \, 1]^d)} = |\Delta_{k_0,j_0}|^{-1/p} \prod_{i=1}^d \lambda_{k_0,i}^{\alpha_i+\beta_i} \|\partial_\beta\partial_\alpha f\| _{L_p(\Delta_{k_0,j_0})}.
$$
Hence
$$
M \underset{p,r,d}{\lesssim} \prod_{i=1}^d \lambda_{k_0,i}^{\beta_i} \|\nabla^r f\| _{L_p(\Delta_{k_0,j_0})} \stackrel{\eqref{lam_ki_est}}{\underset{\mathfrak{Z}}{\lesssim}} 2^{-k_0(r-l)}\|\nabla^r f\| _{L_p(\Delta_{k_0,j_0})}.
$$
We obtain
\begin{align}
\label{der_appr} \|\nabla^l(f - P_{k_0,j_0}f) \| _{L_p(\Delta _{k_0,j_0}^{1/2,1})} \underset{\mathfrak{Z}}{\lesssim} 2^{-k_0(r-l)}\|\nabla^r f\| _{L_p(\Delta_{k_0,j_0})}.
\end{align}

Let $\varphi_0:[0, \, 1] \rightarrow [0, \, 1]$ be an infinitely smooth function, $\varphi_0|_{[0, \, 1/2]} = 0$, $\varphi_0|_{[3/4, \, 1]} = 1$. We set $\varphi(x', \, x_d) = \varphi_0(x_d)$, 
$$
\varphi_{{\cal A}'}(x) = \begin{cases} \varphi(T_{\xi_{k_0,j_0}}x), & x\in \Delta_{k_0,j_0}, \\ 1, & x\in \Omega_{{\cal A}'} \backslash \Delta_{k_0,j_0}. \end{cases}
$$
We have
$$
\|f-P_{k_0,j_0}f\|_{L_q(\Omega_{{\cal A}'})} \le \| \varphi _{{\cal A}'} (f-P_{k_0,j_0}f) \|_{L_q(\Omega_{{\cal A}'})} + 
\| (1-\varphi _{{\cal A}'}) (f-P_{k_0,j_0}f) \|_{L_q(\Delta_{k_0,j_0})}.
$$
Since $\varphi _{{\cal A}'} (f-P_{k_0,j_0}f)|_{\Delta_{k_0,j_0}^{0,1/2}} = 0$ and $(1-\varphi _{{\cal A}'}) (f-P_{k_0,j_0}f)|_{\Delta_{k_0,j_0}^{3/4,1}} = 0$, we have by \eqref{nul_est}, the Sobolev embedding theorem, the change-of-variable formula and \eqref{lam_ki_est}
$$
\| \varphi _{{\cal A}'} (f-P_{k_0,j_0}f) \|_{L_q(\Delta_{k_0,j_0})} \underset{\mathfrak{Z}}{\lesssim} 2^{-k_0(r + 1/q-1/p)}\prod _{i=1}^{d-1} (\varphi_i(2^{-k_0}))^{1/q-1/p} \|\nabla^r (\varphi _{{\cal A}'}(f-P_{k_0,j_0}f))\|_{L_p(\Omega_{{\cal A}'})},
$$
$$
\| (1-\varphi _{{\cal A}'}) (f-P_{k_0,j_0}f) \|_{L_q(\Delta_{k_0,j_0})} \underset{\mathfrak{Z}}{\lesssim} $$$$\lesssim 2^{-k_0(r + 1/q-1/p)}\prod _{i=1}^{d-1} (\varphi_i(2^{-k_0}))^{1/q-1/p} \|\nabla^r ((1-\varphi _{{\cal A}'})(f-P_{k_0,j_0}f))\|_{L_p(\Delta_{k,_0,j_0})}.
$$
Further,
$$
\|\nabla^r (\varphi _{{\cal A}'}(f-P_{k_0,j_0}f))\|_{L_p(\Omega_{{\cal A}'})} \underset{p,r,d}{\lesssim} \|\nabla^r f\| _{L_p(\Omega_{{\cal A}'})} +$$$$+ \sum \limits _{l=0}^{r-1} \|\nabla^{r-l} \varphi _{{\cal A}'}\| _{C(\Delta_{k_0,j_0}^{1/2,3/4})} \|\nabla^l (f-P_{k_0,j_0}f)\| _{L_p(\Delta_{k_0,j_0}^{1/2,3/4})} \stackrel{\eqref{lam_ki_est}, \eqref{der_appr}}{\underset{\mathfrak{Z}}{\lesssim}}
$$
$$
\lesssim \|\nabla^r f\| _{L_p(\Omega_{{\cal A}'})} + \sum \limits _{l=0}^{r-1} 2^{k_0(r-l)} 2^{-k_0(r-l)} \|\nabla^r f\| _{L_p(\Omega_{{\cal A}'})} \underset{p,r,d}{\lesssim} \|\nabla^r f\| _{L_p(\Omega_{{\cal A}'})}
$$
(in estimating $\|\nabla^{r-l} \varphi _{{\cal A}'}\| _{C(\Delta_{k_0,j_0}^{1/2,3/4})}$ we take into account that the function $\varphi _{{\cal A}'}|_{\Delta_{k_0,j_0}}$ depends only on $x_d$).
The same estimate holds for $\|\nabla^r ((1-\varphi _{{\cal A}'})(f-P_{k_0,j_0}f))\|_{L_p(\Delta_{k,_0,j_0})}$.

This completes the proof of \eqref{appr_est1}.

Now we construct the linear continuous projection $Q_{k_0,j_0}$. Let $\{\tilde\eta_j\}_{j\in J}$ be an orthogonal basis in ${\cal P}_{r-1}([0, \, 1]^d)$ with respect to the norm $L_2([0, \, 1]^d)$, $\eta_j(x) = \tilde \eta_j(T_{\xi_{k_0,j_0}}x)$, $x\in \Delta_{k_0,j_0}$. We set for $f\in L_q(\Omega_{{\cal A}'})$
$$
\hat Q_{k_0,j_0} f = \sum \limits _{j\in J} \frac{\eta_j}{\|\eta_j\|^2_{L_2(\Delta_{k_0,j_0})}} \int \limits _{\Delta_{k_0,j_0}}f(x) \eta_j(x)\, dx;
$$
the function $Q_{k_0,j_0} f$ is the extending of the polynomial $\hat Q_{k_0,j_0} f$ to $\Omega$. From \eqref{c_pm} it follows that $\Omega_{{\cal A}'} \subset \Delta'_{k_0,j_0} \times (c^-_{k_0,j_0}, \, c^-_{k_0,j_0}+\tau (c^+_{k_0,j_0}-c^-_{k_0,j_0}))$, where $\tau \underset{\mathfrak{Z}}{\lesssim} 1$. Hence
$$
\|Q_{k_0,j_0} f\| _{L_q(\Omega_{{\cal A}'})} \underset{q, r, d}{\lesssim} \|\hat Q_{k_0,j_0} f\| _{L_q(\Delta_{k_0,j_0})} \le
$$
$$
\le \sum \limits _{j\in J} \|\eta_j\|^{-2} _{L_2(\Delta_{k_0,j_0})} \|\eta_j\| _{L_q(\Delta_{k_0,j_0})} \|\eta_j\| _{L_\infty(\Delta_{k_0,j_0})} \|f\|_{L_1(\Delta_{k_0,j_0})} \le
$$
$$
\le \sum \limits _{j\in J} \|\eta_j\|^{-2} _{L_2(\Delta_{k_0,j_0})} \|\eta_j\|^2 _{L_\infty(\Delta_{k_0,j_0})} |\Delta_{k_0,j_0}|\cdot \|f\|_{L_q(\Delta_{k_0,j_0})}\underset{r, \, d}{\lesssim} \|f\|_{L_q(\Delta_{k_0,j_0})}.
$$
Therefore, 
\begin{align}
\label{q_k0j0}
\|Q_{k_0,j_0}\|_{L_q(\Omega_{{\cal A}'}) \rightarrow L_q(\Omega_{{\cal A}'})} \underset{q,r,d}{\lesssim} 1.
\end{align}
Similarly we prove that the operator $Q_{k_0,j_0}$ is bounded on $L_q(\Omega)$.

Taking into account that for $f\in {\cal W}^r_p(\Omega) \cap C^\infty(\Omega)$ we have $P_{k_0,j_0}f\in {\cal P}_{r-1}(\Omega)$ and that $Q_{k_0,j_0}$ is a projection onto ${\cal P}_{r-1}(\Omega)$, we get
$$
\|f - Q_{k_0,j_0} f\|_{L_q(\Omega_{{\cal A}'})} \le \|f - P_{k_0,j_0} f\|_{L_q(\Omega_{{\cal A}'})} + \|P_{k_0,j_0}f - Q_{k_0,j_0} f\|_{L_q(\Omega_{{\cal A}'})} =
$$
$$
= \|f - P_{k_0,j_0} f\|_{L_q(\Omega_{{\cal A}'})} + \|Q_{k_0,j_0}P_{k_0,j_0}f - Q_{k_0,j_0} f\|_{L_q(\Omega_{{\cal A}'})} \stackrel{\eqref{q_k0j0}}{\underset{q,r,d}{\lesssim}}
$$
$$
\lesssim \|f - P_{k_0,j_0} f\|_{L_q(\Omega_{{\cal A}'})} \stackrel{\eqref{appr_est1}}{\underset{\mathfrak{Z}}{\lesssim}} 2^{-k_0(r + 1/q-1/p)}\prod _{i=1}^{d-1} (\varphi_i(2^{-k_0}))^{1/q-1/p} \|\nabla^r f\|_{L_p(\Omega_{{\cal A}'})}.
$$
This completes the proof.
\end{proof}

\renewcommand{\proofname}{\bf Proof of Theorems \ref{main_entr}, \ref{main_widths}}

\begin{proof}
We check that Assumptions \ref{sup1}--\ref{sup3} hold with $$\lambda_*= r+(\sigma(d-1)+1)(1/q-1/p),\quad u_*(x) = (\psi_{\Lambda}(x))^{1/p-1/q},$$ $$\gamma _* = \sigma(d-1), \quad \psi_*(x) = \psi_\Lambda(x), \quad \delta_*= \frac rd + \frac 1q -\frac 1p.$$

From the conditions of Theorems \ref{main_entr}, \ref{main_widths} it follows that the conditions of Lemma \ref{emb_est1} hold. Indeed, \eqref{rdpdq_g0} follows from \eqref{emb_hol_cond}, since $\sigma\ge 1$. The sequence \eqref{est_seq} equals to $2^{-k(r+(\sigma(d-1)+1)(1/q-1/p))}[\Lambda(2^{-k})]^{1/q-1/p}$, $k\in \N$; this together with Lemma \ref{slow_gr} implies that it is non-increasing for large $k$, and, for $q=\infty$, \eqref{geom_pr_infty} holds.

Assumption \ref{sup1} and \eqref{w_s_2} follow from Lemma \ref{lem2est}; here ${\bf V}({\cal A}_{t,i}) = \{\xi_{t,i}\}$, $\hat J_t = {\bf W} _t$ (then ${\cal A}_{t',i'}$ follows ${\cal A}_{t,i}$ if and only if $\xi_{t',i'}\in {\bf V}_1(\xi_{t,i})$). Relation \eqref{nu_t_k} holds since
$$
\# {\bf W}_k \stackrel{\eqref{del_pr_kj}}{=} \prod_{i=1}^{d-1} 2^{n_{k,i}} \stackrel{\eqref{phi_prop}, \eqref{nki_def}}{\underset{a_*,d}{\asymp}} \prod_{i=1}^{d-1} (\varphi_i(2^{-k}))^{-1} \stackrel{\eqref{prod_phi_i}}{=} \frac{2^{k\sigma(d-1)}}{\Lambda(2^{-k})}.
$$
Relation \eqref{2l} follows from the inequality $r + (\sigma(d-1)+1)(1/q-1/p) >0$. Properties 2 and 3 after \eqref{2l} follow from the definition of ${\cal A}_{t,i}$ and \eqref{card_v1}.

We prove that Assumption \ref{sup2} holds.
We split the parallelepiped $\Delta_{k,j}$ into $2^{md}$ similar equal parallelepipeds  $\Delta_{k,j,t}$, $1\le t\le 2^{md}$. From the embedding theorem, \eqref{lam_ki_est}, \eqref{fkj_vol} and the change-of-variable formula it follows that there are linear continuous operators $P_{k,j,t}: L_q(\Omega) \rightarrow {\cal P}_{r-1}(\Delta_{k,j,t})$ such that for $f \in {\cal W}^r_p(\Omega)$
$$
\|f- P_{k,j,t}f\| _{L_q(\Delta_{k,j,t})} \underset{\mathfrak{Z}}{\lesssim} \Bigl(\prod _{i=1} ^{d-1} \varphi_i(2^{-k})\cdot 2^{-m}\Bigr) ^{1/q-1/p} (2^{-k-m}) ^{1/q-1/p} \cdot 2^{-(k+m)r} \|\nabla^r f\|_{L_p(\Delta_{k,j,t})} \stackrel{\eqref{prod_phi_i}}{=}
$$
$$
=2^{-k(r+(\sigma(d-1)+1)(1/q-1/p))} (\psi_{\Lambda}(2^k))^{1/p-1/q}\cdot 2^{-m(r+d/q-d/p)} \|\nabla^r f\|_{L_p(\Delta_{k,j,t})}.
$$

Now Theorem \ref{trm2} implies Theorem \ref{main_widths}. Theorem \ref{trm1} yields the estimates for the entropy numbers of the embedding operator of the space $\hat {\cal W}^r_p(\Omega) = \{f - Q_{0,1}f:\; f\in {\cal W}^r_p(\Omega)\}$; it implies the estimate in Theorem \ref{main_entr}, since $\{Q_{0,1}f:\; \|f\| _{L_p(\Omega)} \le 1\}$ is a bounded subset in a finite-dimensional space ${\cal P}_{r-1}(\Omega)$; hence the sequence $e_n(Q_{0,1}: \tilde {\cal W}^r_p(\Omega) \rightarrow L_q(\Omega))$ decreases faster than 
a~power-law sequence due to estimates on the entropy numbers of 
finite-dimensional balls  \cite{schutt}.
\end{proof}
\renewcommand{\proofname}{\bf Proof}

\section{Proof of Theorem \ref{h_set_e_w} and the upper estimate in Example \ref{exa1}}

We first consider the more general case of the functions $h$ and $\psi$ and obtain analogues of Lemmas \ref{emb_est1}, \ref{lem2est}.

Suppose that
\begin{enumerate}
\item $h:[0, \, 1] \rightarrow [0, \, \infty)$ is a non-decreasing function, $h(t)>0$ for $t>0$, $h(2t)\le a_*h(t)$, where $a_*\ge 1$;

\item the function $\psi$ is given by the equation
\begin{align}
\label{psi_psi0_dist}
\psi(x') = 2 - \psi_0({\rm dist}(x', \, \Gamma)), 
\end{align}
where $\psi_0:[0, \, 1] \rightarrow [0, \, 1]$ is a strictly increasing function, $\psi_0|_{[2^{-\nu}, \, 2^{-\nu+1}]}$ is Lipschitz with the constant $b_*\cdot 2^{\nu} \psi_0(2^{-\nu})$ (here $b_*>0$ does not depend on $\nu$); the inverse function $\varphi_0 = \psi_0^{-1}$ satisfies the conditions $\varphi_0(t) \le a_*t$, $\varphi_0(2t) \le a_* \varphi_0(t)$. We also suppose that the conditions of Lemma \ref{emb_est1} hold with $\varphi_1=\dots = \varphi_{d-1}:= \varphi_0$.
\end{enumerate}

We argue as in \S 2, 3 with $\varphi_1=\dots = \varphi_{d-1}:= \varphi_0$. The numbers $n_k\in \N$ are defined by the condition similar to \eqref{nki_def}:
\begin{align}
\label{2nkphi0k3}
2^{-n_k+1} \le \varphi_0(2^{-k-4}) < 2^{-n_k+2}.
\end{align}
Consider the partition of $(0, \, 1)^{d-1}$ into open cubes 
\begin{align}
\label{del_pr_kj1}
\Delta'_{k,j} = \prod_{i=1}^{d-1} (\zeta_{k,j,i}, \, \zeta_{k,j,i} + 2^{-n_k}), \quad j\in {\bf W}_k.
\end{align}
The sets ${\bf W}_{k-1,j}$ are defined according to \eqref{w_kj}.

Given sets $A$, $B\subset \R^{d-1}$, we denote
$$
{\rm dist}(A, \, B) = \inf _{x'\in A, \, y'\in B}\|x'-y'\|_{l_\infty ^{d-1}}.
$$

Let
\begin{align}
\label{jk1jk2}
J_{k,1} = \{j\in {\bf W}_k:\; {\rm dist}(\Delta'_{k,j}, \, \Gamma)\le 2^{-n_k}\}, \quad J_{k,2} = {\bf W}_k \backslash J_{k,1}.
\end{align}

The numbers $c^{\pm}_{k,l}$ ($l\in {\bf W}_{k-1,j}$, $j\in J_{k-1,1}$) are defined according to \eqref{c_kl_pm}. We set
$$
U_{k,l} = \{(x', \, x_d):\; x'\in \Delta'_{k,l}, \; c^+_{k,l}\le x_d< \psi(x')\}.
$$

Notice that from the definition of numbers $c^{+}_{k,l}$ it follows that
\begin{align}
\label{inf_psi} \inf _{x'\in \Delta'_{k,l}} \psi(x')= c^+_{k,l} + 2^{-k-1}.
\end{align}

We claim that
\begin{align}
\label{cpmkl2k} c^+_{k,l} - c^-_{k,l} \asymp 2^{-k}.
\end{align}
Indeed, as in proof of \eqref{c_pm} we get that $c^+_{k,l} - c^-_{k,l} = 2^{-k-1} -\psi(x'_-) + \psi(x'_+)$, where $x'_{\pm} \in \overline{\Delta}'_{k-1,j}$, $j\in J_{k-1,1}$. Hence it suffices to check that $|\psi(x'_-) - \psi(x'_+)|\le 2^{-k-2}$. From \eqref{psi_psi0_dist} we get
$$
|\psi(x'_-) - \psi(x'_+)| \le \psi_0({\rm dist}(x'_-, \, \Gamma)) + \psi_0({\rm dist}(x'_+, \, \Gamma)) \stackrel{\eqref{del_pr_kj1}, \eqref{jk1jk2}}{\le}$$$$\le 2\psi_0(2^{-n_{k-1}+1}) \stackrel{\eqref{2nkphi0k3}}{\le} 2^{-k-2}.
$$

Let $j\in J_{k,2}$, $j\in {\bf W}_{k-1,t}$, $t\in J_{k-1,1}$. Denote by $L_{k,j}$ the Lipschitz constant of the function $\psi|_{\Delta'_{k,j}}$, and by $L_{0,k,j}$, the Lipschitz constant of the function $\psi_0|_{[2^{-n_k}, \, 2^{-n_{k-1}+1}]}$. Then, for any $x'\in \Delta'_{k,j}$, we have ${\rm dist}(x', \, \Gamma)\in [2^{-n_k}, \, 2^{-n_{k-1}+1}]$; therefore, $L_{k,j}\le L_{0,k,j}$, and by the monotonicity of the functions $\psi_0$ and $\varphi_0$ we get
$$
L_{0,k,j}\stackrel{\eqref{psi_psi0_dist}}{\le} \max _{n_{k-1}\le \nu \le n_k} b_*\cdot 2^{\nu} \psi_0(2^{-\nu}) \le b_*\cdot 2^{n_k} \psi_0(2^{-n_{k-1}}) \stackrel{\eqref{2nkphi0k3}}{\le} b_*\cdot 2^{n_k} \psi_0(\varphi_0(2^{-k-3})) \lesssim b_*\cdot 2^{n_k-k}.
$$
Hence
\begin{align}
\label{lkj_est} L_{k,j} \underset{b_*}{\lesssim} 2^{n_k-k}.
\end{align}

Let ${\cal A}$ be the tree defined in \S 2. Consider its subtree $\tilde{\cal A}$ with the vertex set $\{\xi_0\}\cup \{\xi_{k,l}:\; k\in \N, \, l\in {\bf W}_{k-1,j}, \, j\in J_{k-1,1}\}$. We define the mapping $\hat F$ on the set ${\bf V}(\tilde {\cal A})$ as follows:
$$
\hat F(\xi_{k,l}) = \begin{cases} \Delta_{k,l}, & l\in J_{k,1}, \\ \Delta_{k,l} \cup U_{k,l}, & l\in J_{k,2}. \end{cases}
$$
Then $\{\hat F(\xi)\} _{\xi\in {\bf V}(\tilde {\cal A})}$ is a partition of $\Omega$. The arguments are almost the same as in \S 2. We use notation from the proof. In the case $x_d> c^+_{k,l}$ it suffices to condider only $l\in J_{k,1}$, and in the case $x_d>c^+_{k+1,t}$, only $t\in J_{k+1,1}$. We show that the last case is impossible. Indeed, there is a point $x'_+\in \overline{\Delta}'_{k+1,t}$ such that
$$
x_d> \psi(x'_+) - 2^{-k-2}\ge \psi(x') - 2^{-k-2} - |\psi(x') - \psi(x'_+)| \ge
$$
$$
\ge \psi(x') - 2^{-k-2}-\psi_0({\rm dist}(x', \, \Gamma))- \psi_0({\rm dist}(x'_+, \, \Gamma)).
$$
It suffices to show that
$$
\psi_0({\rm dist}(x', \, \Gamma))\le 2^{-k-4}, \quad \psi_0({\rm dist}(x'_+, \, \Gamma)) \le 2^{-k-4}.
$$
We check the first inequality (the second one is similar). Since $x'\in \overline{\Delta}'_{k+1,t}$ and $t\in J_{k+1,1}$, we have by \eqref{del_pr_kj1}, \eqref{jk1jk2}
$$
\psi_0({\rm dist}(x', \, \Gamma)) \le \psi_0(2^{-n_{k+1}+1}) \stackrel{\eqref{2nkphi0k3}}{\le} \psi_0(\varphi_0(2^{-k-5})) = 2^{-k-5}.
$$

Denote $\mathfrak{Z}_1 = (p, \, q, \, r, \, d, \, h, \, \psi_0, \, c_*)$, where $c_*$ is from Definition \ref{def_h_set}.

The estimate of $\|f-Q_{k_0,j_0}f\|_{L_p(\Omega_{{\cal A}'})}$ can be proved as in \S 3 (see Lemmas \ref{emb_est1}, \ref{lem2est}); it has the form
\begin{align}
\label{t3pr}
\|f-Q_{k_0,j_0}f\|_{L_q(\Omega_{{\cal A}'})} \underset{\mathfrak{Z}_1}{\lesssim} 2^{-k_0(r+1/q-1/p)} (\varphi_0(2^{-k_0})) ^{(1/q-1/p)(d-1)} \|\nabla^r f\| _{L_p(\Omega_{{\cal A}'})}.
\end{align}
In obtaining the analogue of Lemma \ref{emb_est1} we modify the definition of the polygonal lines $\gamma_{x,x_0}$ and of the numbers $r_i(t)$. They are constructed as follows (the notation is taken from the proof of Lemma \ref{emb_est1}).

Let $x\in \hat F(\xi_{k_0+s,j_s})$. If $j_s\in J_{k_0+s,1}$ or $x\in \Delta_{k_0+s,j_s}$ with $j_s\in J_{k_0+s,2}$, then $\gamma_{x,x_0}$ and $r_i(t)$ are defined as in the proof of Lemma \ref{emb_est1}.

Consider the case $j_s\in J_{k_0+s,2}$, $x\in U_{k_0+s,j_s}$. Let $\varkappa\in \N$ (this number will be chosen later in dependence of $d$, $a_*$, $b_*$). We split $\Delta'_{k_0+s,j_s}$ into $2^{\varkappa (d-1)}$ same cubes $\Delta'_{k_0+s,j_s,m}$, $1\le m\le 2^{\varkappa d}$. Let $x'_{k_0+s,j_s,m}$ be the center of the cube $\Delta'_{k_0+s,j_s,m}$.

Let $x = (x', \, x_d)\in U_{k_0+s,j_s}$, $x'\in \Delta'_{k_0+s,j_s,m}$. We set $$\tilde x_{(s)} = \Bigl(x'_{k_0+s,j_s,m}, \, \frac{c^-_{k_0+s,j_s} + c^+ _{k_0+s,j_s}}{2}\Bigr)\in \Delta_{k_0+s,j_s},$$
$$\gamma_{x,x_0}(t) = \begin{cases} x\frac{|x-\tilde x_{(s)}|-t}{|x-\tilde x_{(s)}|}+ \tilde x_{(s)}\frac{t}{|x-\tilde x_{(s)}|}, & 0\le t\le |x-\tilde x_{(s)}|, \\ \gamma_{\tilde x_{(s)},x_0}(t-|x-\tilde x_{(s)}|), & |x-\tilde x_{(s)}|\le t\le T_{\tilde x_{(s)},x_0} + |x-\tilde x_{(s)}|.\end{cases}$$
The number $\tilde t_{(s)}$ is defined by the equation $\gamma_{x,x_0}(\tilde t_{(s)}) = \tilde x_{(s)}$.

For $i=1, \, \dots, \, d-1$ we set
$$
r_i(t) = \begin{cases} 0, & t = 0, \\ 2^{-n_{k_0+s}}, & t = \tilde t_{(s)}, \\ 2^{-n_{k_0+s-1}}, & t = \tilde t_{(s-1)}, \\ 2^{-n_{k_0+l}}, & t = t_{(l)}, \; 1\le l\le s-1, \\ 2^{-n_{k_0}}, & t = T_{x,x_0},\end{cases}
$$
$$
r_d(t) = \begin{cases} 0, & t = 0, \\ 2^{-k_0-s}, & t = \tilde t_{(s)}, \\ 2^{-k_0-s+1}, & t = \tilde t_{(s-1)}, \\ 2^{-k_0-l}, & t = t_{(l)}, \; 1\le l\le s-1, \\ 2^{-k_0}, & t = T_{x,x_0};\end{cases}
$$
after that we interpolate $r_i(t)$ as the piecewise-affine functions on $[0, \, T_{x,x_0}]$ with knots at $0, \, \tilde t_{(s)}, \, \tilde t_{(s-1)}, \, t_{(s-1)}, \, t_{(s-2)}, \,\dots, \, t_{(1)}, \, T_{x,x_0}$.

The number $\varkappa = \varkappa(d, \, a_*, \, b_*)$ is chosen sufficiently large, and $c = c(d, \, a_*)$ is sufficiently small; then the inclusion \eqref{incl_c} and the inequalities \eqref{dot_gamma} hold, as well as \eqref{t_y_bxx0} and the estimates \eqref{t_por_xmy}, \eqref{max_xi} for $y\in \hat F(\xi_{k_0+s-1,j_{s-1}})\cup \hat F(\xi_{k_0+s,j_s})$ (here we use \eqref{lkj_est}).

After that we repeat arguments from Lemmas \ref{emb_est1}, \ref{lem2est}, taking into account \eqref{inf_psi}, \eqref{cpmkl2k} and \eqref{lkj_est}, and obtain \eqref{t3pr}.

We show that
\begin{align}
\label{jk1} \# J_{k,1} \underset{\mathfrak{Z}_1}{\lesssim} \frac{1}{h(2^{-n_k})}.
\end{align}

Let $\Gamma_k = \{y_i:\; 1\le i\le N_k\}$ be a subset in $\Gamma$ of maximal cardinality such that $\|y_i-y_{i'}\|_{l_\infty^{d-1}} \ge 2^{-n_k+2}$, $i\ne i'$. Then $\Gamma_k$ is a $2^{-n_k+2}$-net for $\Gamma$. Let $B_i$ be the balls with respect to the $l_\infty^{d-1}$-norm with centers at $y_i$ and radii $2^{-n_k}$. They do not intersect pairwise and $\mu(B_i) \stackrel{\eqref{h_set_def}}{\underset{c_*,d}{\asymp}} h(2^{-n_k})$. Hence 
\begin{align}
\label{nk_1h}
N_k \underset{\mathfrak{Z}_1}{\lesssim} \frac{1}{h(2^{-n_k})}.
\end{align}

Given $j\in J_{k,1}$, we take an index $i(j)$ such that $y_{i(j)}$ is the nearest point to $\Delta'_{k,j}$ from $\Gamma_k$; i.e. ${\rm dist}(\Delta'_{k,j}, \, \Gamma_k) = {\rm dist}(y_{i(j)}, \, \Delta'_{k,j})$. Since $\Gamma_k$ is the $2^{-n_k+2}$-net for $\Gamma$, we get by \eqref{jk1jk2} that ${\rm dist}(\Delta'_{k,j}, \, \Gamma_k) \lesssim 2^{-n_k}$. Hence, since the length of the edge in $\Delta'_{k,j}$ is $2^{-n_k}$, we get that $\{j\in J_{k,1}:\; i(j)=i\}\underset{d}{\lesssim} 1$ for each $i\in \{1, \, \dots, \, N_k\}$. This together with \eqref{nk_1h} implies \eqref{jk1}.

Under the conditions of Theorem \ref{h_set_e_w}, we have $h(t) = t^\theta$, $\psi_0(t) = t^{1/\sigma}$, $\varphi_0(t) = t^\sigma$. From the definition of $\tilde {\cal A}$, \eqref{card_v1}, \eqref{2nkphi0k3}, \eqref{t3pr} and \eqref{jk1} it follows that
Assumptions \ref{sup1} and \ref{sup3} hold with $\lambda_* = r+(\sigma(d-1)+1)(1/q-1/p)$, $\gamma_* = \sigma \theta$, $u_*\equiv 1$, $\psi_* \equiv 1$.

Now we check that Assumption \ref{sup2} holds. Indeed, let $T_{\xi_{k,l}}$ be the affine operator from Notation \ref{notation}. Then $T_{\xi_{k,l}}(\hat F(\xi_{k,l}))$ has the form
$$
G:=\{(x', \, x_d):\; x'\in (0, \, 1)^{d-1}, \; 0< x_d< \tilde \psi(x')\},
$$
where $\tilde \psi'$ is Lipschitz with the constant $C = C(\mathfrak{Z}_1)$, $\inf _{x'\in (0, \, 1)^{d-1}} \tilde \psi(x') \in [1, \, C_1]$, where $C_1 = C_1(\mathfrak{Z}_1)\ge 1$. For each $m\in \Z_+$, $n\in \N$ there is a partition $T_{m,n}$ of the set $G$ such that the conditions of Assumption \ref{sup2} hold with $w_* \equiv c_2 = c_2(\mathfrak{Z}_1)$, $\delta_* = \frac rd + \frac 1q - \frac 1p$ (it follows, e.g., from \cite[Lemma 8]{vas_john}). Applying the change-of-variable formula, we get that Assumption \ref{sup2} holds with $w_*(\xi_{k,l}) = 2^{-k(r +(\sigma(d-1)+1)(1/q-1/p))}$, $\delta_* = \frac rd + \frac 1q - \frac 1p$.

It remains to apply Theorems \ref{trm1}, \ref{trm2}.

\vskip 0.3cm

In Example \ref{exa1} we argue similarly, taking $\varphi_1(t)=\dots =\varphi_\theta(t) = t$, $\varphi_{\theta+1}(t) = \dots = \varphi_{d-1}(t) = t^\sigma$. We obtain that Assumptions \ref{sup1}--\ref{sup3} hold with $\lambda_* = r+(1/q-1/p)(\theta+1 +\sigma(d-\theta-1))$, $\gamma_* = \theta$, $u_*\equiv 1$, $\psi_* \equiv 1$, $\delta_* = \frac rd + \frac 1q - \frac 1p$.

\section{The example for the lower estimate}

Here for convenience we consider the cube $Q = [-1/2, \, 1/2]^{d-1}$ instead of $[0, \, 1]^{d-1}$. The function $\psi$ is the same as in Theorem \ref{h_set_e_w}; i.e., 
\begin{align}
\label{psi_x_2mdist}
\psi(x') = 2 -({\rm dist}(x', \, \Gamma))^{1/\sigma},
\end{align}
where $\sigma \ge 1$,
$$
\Omega =\{(x', \, x_d):\; x'\in (-1/2, \, 1/2)^{d-1}, \; 0<x_d<\psi(x')\}.
$$
Here $\Gamma$ is the Cantor-type $h$-set constructed in \cite{bricchi}, with $h(t) = t^\theta$, $0<\theta < d$. We briefly recall the construction.

Let $h(t) = t^\theta$, $0<\theta<d$. Then there exist a sequence of positive numbers $\{\lambda_k\}_{k\in \N}$ and a number $m = m(\theta, \, d)\in \N$, $m\ge 2$, such that 
\begin{gather}
\label{h_lll}
h(\lambda_1\lambda_2\dots \lambda_k) \underset{\theta,d}{\asymp} m^{-(d-1)k}, 
\\
\label{inf_k_lk}
0< \inf _k \lambda_k\le \sup _k \lambda_k <m^{-1}
\end{gather}
(see \cite[proof of Theorem 7.4]{bricchi} with $n:=d-1$).

From \eqref{h_lll} it follows that, for $h(t) = t^\theta$, 
\begin{align}
\label{h_lll1}
\lambda_1\lambda_2\dots \lambda_k \underset{\theta, d}{\asymp} m^{-(d-1)k/\theta}.
\end{align}

We split $Q$ into $m^{d-1}$ same cubes; denote their centers by $v_j$. We set $f_{k,j}(x') =\lambda_k x'+v_j$, $k\in \N$, $j=1, \, \dots, \, m^{d-1}$. It was proved that the sequence of the sets 
$$
E_k = \cup_{i_1, \, \dots, \, i_k\in \{1, \, \dots, \, m^{d-1}\}} f_{1,i_1} \circ \dots \circ f_{k,i_k} (Q)
$$
converges in the Hausdorff metric to a compact set $\Gamma$, which will be the desired $h$-set (see \cite[proof of Theorem 6.6]{bricchi}). We can see from the definition that $E_{k+1} \subset E_k$ ($k\in \N$) and each cube $f_{1,i_1} \circ \dots \circ f_{k,i_k} (Q)$ contains an element of $\Gamma$.

We show that for such $\Gamma$ the lower estimates for the widths and the entropy numbers have the same order as the upper estimates from Theorem \ref{h_set_e_w}.

The set of parameters $\mathfrak{Z}_1$ is defined as in the previous section.

It suffices to construct functions $\varphi_{k,j} \in \tilde{\cal W}^r_p(\Omega)$, $k\in \N$, $1\le j\le m^{(d-1)k}$ with pairwise non-intersecting supports such that 
\begin{align}
\label{phi_kj} \|\nabla^r \varphi_{k,j}\| _{L_p(\Omega)} + \|\varphi_{k,j}\| _{L_p(\Omega)} \underset{\mathfrak{Z}_1}{\lesssim} m^{k(d-1)\frac{r-(\sigma(d-1)+1)/p}{\theta \sigma}}, \quad \|\varphi_{k,j}\| _{L_q(\Omega)} \underset{\mathfrak{Z}_1}{\gtrsim} m^{-k(d-1)\frac{(\sigma(d-1)+1)/q}{\theta \sigma}}
\end{align}
(the lower estimates for the entropy numbers can be proved as Lemma 12 in \cite{vas_entr}; the lower estimates for the widths follow from Lemma 6 in \cite{vas_width}).

We numerate the cubes $f_{1,i_1} \circ \dots \circ f_{k,i_k} (Q)$ ($1\le i_1, \, \dots, \, i_k\le m^{d-1}$) by indices $j\in \{1, \, \dots, \, m^{(d-1)k}\}$ and denote them by $Q_{k,j}$. Since $\Gamma \subset E_{k+1}$, we have for each boundary point $x'\in \partial Q_{k,j}$ 
$$
{\rm dist}(x', \, \Gamma) \ge \lambda_1\dots \lambda_k \cdot \frac{m^{-1}-\lambda_{k+1}}{2};
$$
hence
$$
\psi(x') \stackrel{\eqref{psi_x_2mdist}}{\le} 2 - (\lambda_1\dots \lambda_k(m^{-1}-\lambda_{k+1})/2)^{1/\sigma} \stackrel{\eqref{inf_k_lk}, \eqref{h_lll1}}{\le} 2-c_1 \cdot m^{-(d-1)k/\theta \sigma},
$$
where $c_1 = c_1(\theta,\, \sigma, \, d)>0$.
We set 
\begin{align}
\label{bk_def}
b_k = 2 -\frac{c_1}{2}\cdot m^{-(d-1)k/\theta \sigma}.
\end{align}
Then $b_k>0$,
\begin{align}
\label{psi_b_k} \psi(x') < b_k, \quad x'\in \partial Q_{k,j}.
\end{align}

Let $\rho \in C^\infty[0, \, \infty)$, $\rho|_{[0, \, 1/2]} = 0$, $\rho|_{[1, \, \infty]} = 1$, $\rho(t)\in [0, \, 1]$ for all $t\in [0, \, \infty)$. Given $k\in \N$, $j\in \{1, \, \dots, \, m^{(d-1)k}\}$, we set $\rho_k(t) = \rho \Bigl( 2\cdot \frac{t-b_k}{2-b_k}\Bigr)$,
$$
\varphi_{k,j}(x', \, x_d) = \begin{cases} 0, & (x', \, x_d)\in \Omega, \; x_d\le b_k \text{ or }x' \notin Q_{k,j}, \\ \rho_k(x_d), & (x', \, x_d)\in \Omega, \;x_d> b_k \text{ and }x' \in Q_{k,j}.\end{cases}
$$
By \eqref{psi_b_k}, $\varphi_{k,j}\in C^\infty(\Omega)$. We claim that \eqref{phi_kj} holds for these functions.

Indeed, for $x'\in Q_{k,j}$, $x_d>b_k$ we have $$|\nabla^r\varphi_{k,j}(x', \, x_d)| \underset{\mathfrak{Z}_1}{\lesssim} \frac{1}{(2-b_k)^r} \stackrel{\eqref{bk_def}}{\underset{\mathfrak{Z}_1}{\lesssim}} m^{k(d-1)r/\theta \sigma},$$
$$
{\rm mes}(Q_{k,j}\times [b_k, \, 2]) \stackrel{\eqref{bk_def}}{\underset{\mathfrak{Z}_1}{\asymp}} (\lambda_1\dots \lambda_k)^{d-1}\cdot m^{-k(d-1)/\theta \sigma} \stackrel{\eqref{h_lll1}}{\underset{\mathfrak{Z}_1}{\asymp}} m^{-k(d-1)((d-1)\sigma+1)/\theta \sigma}.
$$
This implies the first estimate in \eqref{phi_kj}.

Let us prove the second estimate in \eqref{phi_kj}. If 
\begin{align}
\label{2xdbk}
2\frac{x_d-b_k}{2-b_k}\ge 1,
\end{align}
then
\begin{align}
\label{phi11111}
\varphi(x', \, x_d)=1 \text{ for }x'\in Q_{k,j}.
\end{align}
The condition \eqref{2xdbk} is equivalent to $2-x_d \le \frac{c_1}{4} m^{-(d-1)k/\theta \sigma}$. Recall that $\Gamma \cap Q_{k+l,i}\ne \varnothing$ for any $l\in \Z_+$, $i\in \{1, \, 2, \, \dots, \, m^{(d-1)(k+l)}\}$. Hence if $x'\in Q_{k+l,i} \subset Q_{k,j}$, then ${\rm dist}(x', \, \Gamma) \le \lambda_1\dots \lambda_{k+l}$, 
$$
\psi(x') \ge 2 - (\lambda_1\dots \lambda_{k+l})^{1/\sigma} \stackrel{\eqref{h_lll1}}{\ge} 2 - c_2m^{-(d-1)(k+l)/\theta \sigma},
$$
where $c_2 = c_2(\mathfrak{Z}_1)>0$. Therefore, there exists $l = l(\mathfrak{Z}_1)\in \N$ such that the condition $x_d\le \frac{b_k}{4} + \frac 32 \stackrel{\eqref{bk_def}}{=} 2 - \frac{c_1}{8} m^{-k(d-1)/\theta \sigma}$, $x'\in Q_{k+l,i}$ implies the inequality $x_d< \psi(x')$. Hence $Q_{k+l,i}\times [(2+b_k)/2, \, b_k/4 + 3/2]\subset \Omega$,
$$
{\rm mes}(Q_{k+l,i}\times [(2+b_k)/2, \, b_k/4 + 3/2]) \stackrel{\eqref{bk_def}}{\underset{\mathfrak{Z}_1}{\asymp}} $$$$ \asymp (\lambda_1\dots \lambda_{k+l})^{d-1}\cdot m^{-k(d-1)/\theta \sigma} \stackrel{\eqref{h_lll1}}{\underset{\mathfrak{Z}_1}{\asymp}} m^{-k(d-1)((d-1)\sigma+1)/\theta \sigma}.
$$
This together with \eqref{2xdbk}, \eqref{phi11111} implies the second estimate in \eqref{phi_kj}.

\section{Supplement}

In this section we show that the domain from \cite{besov_holder} can be represented as a finite union of domains for which the estimates from Theorems \ref{main_entr}, \ref{main_widths} hold.

Let $\lambda = (\lambda_1, \, \dots, \, \lambda_{d-1}, \, \lambda_d)$, $\lambda_i\ge 1$ for $i=1, \, \dots, \, d-1$, \, $\lambda_d=1$. Let $\eta>0$, $T>0$. Given $0<t\le T$, we denote 
\begin{align}
\label{tl_eta}
t^\lambda [-\eta, \, \eta]^d = \prod _{i=1}^d [-t^{\lambda_i}\eta, \, t^{\lambda_i}\eta], \quad t^\lambda (-\eta, \, \eta)^d = \prod _{i=1}^d (-t^{\lambda_i}\eta, \, t^{\lambda_i}\eta).
\end{align}
Let $e_1, \, \dots, \, e_d$ be the standard basis in $\R^d$. We set
\begin{align}
\label{v_lambda}
V_\lambda = \cup _{0<t\le T} (-te_d + t^\lambda [-\eta, \, \eta]^d).
\end{align}

Recall that the domain $G\subset \R^d$ satisfies cone condition if for each point $x\in G$ there is a right circular cone $V_x$ of fixed aperture angle  and fixed height with vertex at~$0$ such that  $x + V_x\subset G$. Bounded domains with cone condition are John domains; hence the entropy numbers for the embedding operator ${\rm Id}:\tilde {\cal W}^r_p(G)\to L_q(G)$ and the widths $\vartheta_n(\tilde {\cal W}^r_p(G), \, L_q(G))$ have the same orders as for $G = (0, \, 1)^d$ (if the orders are known). 

In \cite{besov_holder} the domain had the form $G = \cup _{j=1}^{j_0} G_j$, where $G_j$ are open bounded sets with the following property: there are an orthogonal operator $R_j$ and numbers $\lambda^{(j)}_i\ge 1$ $(i=1, \, \dots, \, d)$ such that $\lambda^{(j)}_d=1$, $\lambda^{(j)}_1 +\dots + \lambda^{(j)}_d = \Lambda$, $G_j + R_jV_{\lambda^{(j)}} \subset G$.
Notice that in this case
\begin{align}
\label{g_besov}
G = \cup _{j=1}^{j_0} (G_j + R_jV_{\lambda^{(j)}}).
\end{align}
We show that if $G$ has the form \eqref{g_besov}, where $G_j$ are open bounded sets, then $G$ is a finite union of domains with cone condition and domain from the class ${\cal G}'_{\varphi_1, \, \dots, \, \varphi_{d-1}}$ with $\varphi_i(t) = \alpha_j \cdot t^{\lambda_i^{(j)}}$, $i=1, \, \dots, \, d-1$, where $\alpha_j\in (0, \, 1]$ are some constants.

It suffices to consider the case 
\begin{align}
\label{g_u_vl}
G = U + V_\lambda,
\end{align}
where $U\subset \R^d$ is an open bounded set. Moreover, we can cover the set $U$ by translations of the parallelepipeds $T^\lambda (-\eta/10, \, \eta/10)^d$ (see notation \eqref{tl_eta}) and reduce the problem to the case
\begin{align}
\label{u_sbs_t}
U \subset T^\lambda (-\eta/10, \, \eta/10)^d.
\end{align}
If $\eta\ge 1$, then $V_\lambda = -Te_d + T^\lambda [-\eta, \, \eta]^d$ and $U+V_\lambda$ satisfies cone condition.

Further we consider the case $0<\eta<1$.

For $0<t_*< T$ we put
$$
V_\lambda^{t_*} = \cup _{t_*\le t\le T} (-te_d + t^\lambda [-\eta, \, \eta]^d).
$$
Then the set $U + V_\lambda^{t_*}$ is a bounded domain with cone condition.

We denote $\lambda'=(\lambda_1, \, \dots, \, \lambda_{d-1})$.

Let $E$ be the orthogonal projection of the set $U$ to the subspace $\{(x_1, \, \dots, \, x_{d-1}, \, 0):\; x_i\in \R, \; 1\le i\le d\}$. By \eqref{u_sbs_t}, we have $E\subset T^{\lambda'}[-\eta/10, \, \eta/10]^{d-1}$. Let
\begin{gather}
\label{pi_def_dop}
\Pi = T^{\lambda'}(-\eta/5, \, \eta/5)^{d-1}, \\ 
\label{til_g_def_dop}
\tilde G = \{(x', \, x_d)\in U+V_\lambda:\; x'\in \Pi\}.
\end{gather}
Then
\begin{align}
\label{pi_sbs_xi}
\Pi \subset \xi' + T^{\lambda'}[-\eta, \, \eta]^{d-1}, \quad \xi'\in E.
\end{align}

We denote 
\begin{align}
\label{m_pm}
M_- = \inf \{\xi_d-T -T\eta:\; \xi\in U\}, \quad M_+ = \sup \{\xi_d-T -T\eta:\; \xi\in U\}.
\end{align}
Then
\begin{align}
\label{mpmmm}
M_+-M_- \stackrel{\eqref{u_sbs_t}}{\le} T\eta/5.
\end{align}

We show that 
\begin{align}
\label{pi_times_g}
\Pi \times (M_-, \, M_+]\subset \tilde G.
\end{align}
Indeed, let $\varepsilon>0$, $\xi=(\xi', \, \xi_d)\in U$, \, $\xi_d-T-T\eta\le M_-+\varepsilon$. Then
$$
\xi+ V_\lambda \stackrel{\eqref{v_lambda}}{\supset} \xi+(-Te_d+T^\lambda[-\eta, \, \eta]^d) \stackrel{\eqref{pi_sbs_xi}}{\supset} \Pi \times [\xi_d - T -T\eta, \, \xi_d - T+T\eta] \supset
$$
$$
\supset \Pi \times [M_-+\varepsilon, \, M_-+2T\eta] \stackrel{\eqref{mpmmm}}{\supset} \Pi \times [M_-+\varepsilon, \, M_+].
$$

Given $x'\in \Pi$, we put
\begin{align}
\label{psi_def_dop}
\Psi(x') = \sup \{c\in \R:\; (x', \, c)\in \tilde G\}
\end{align}
(the set of such numbers $c$ is non-empty, since by  \eqref{til_g_def_dop} and \eqref{pi_sbs_xi} for all $x'\in \Pi$, $\xi'\in E$ we have $(x', \, \xi_d-Te_d)\in (\xi', \, \xi_d)-Te_d+T^\lambda[-\eta, \, \eta]^d$).

We claim that
\begin{align}
\label{podgr}
\tilde G = \{(x', \, x_d):\; x'\in \Pi, \; M_-< x_d <\Psi(x')\}.
\end{align}
The inclusion $\tilde G \subset \{(x', \, x_d):\; x'\in \Pi, \; M_-< x_d <\Psi(x')\}$ follows from \eqref{v_lambda}, \eqref{til_g_def_dop}, \eqref{m_pm} and \eqref{psi_def_dop}. Let us prove the inverse inclusion. Let $x'\in \Pi$.
We show that
\begin{align}
\label{x_pr_times_m}
\{x'\} \times (M_-, \, \Psi(x'))\subset \tilde G.
\end{align}
We take an arbitrary number $c< \Psi(x')$ such that $(x', \, c)\in \tilde G$ and check that
\begin{align}
\label{x_pr_times_m1}
\{x'\} \times (M_-, \, c]\subset \tilde G.
\end{align}
There exists a point $(\xi', \, \xi_d)\in U$ such that $(x', \, c)\in (\xi', \, \xi_d) + V_\lambda$; hence $c = \xi_d + tz\eta - t$ for some $t\in (0, \, T]$, $z\in [-1, \, 1]$. By increasing  $c$ if necessary, we may assume that $z=1$. Further, $x'\in \xi'+t^{\lambda'}[-\eta, \, \eta]^{d-1}$; therefore $$\{x'\}\times [\xi_d-t\eta -t, \, c]\subset (\xi', \, \xi_d)+t^\lambda[-\eta, \, \eta]^d \subset \tilde G,$$ $(x', \, \xi_d-s\eta -s)\in \tilde G$ for any $s\in [t, \, T]$. Hence $\{x'\}\times [\xi_d-T\eta -T, \, c] \subset \tilde G$; from the definition of $M_+$ we get $\{x'\} \times [M_+, \, c]\subset \tilde G$. This together with \eqref{pi_times_g} implies \eqref{x_pr_times_m1}. Since $c$ can be arbitrarily close to $\Psi(x')$, we obtain \eqref{x_pr_times_m}. This completes the proof of the inclusion \eqref{podgr}.

We claim that for $0<t\le 1$
\begin{align}
\label{g_hol}
\text{if }x',\; y'\in \Pi, \; |x_i-y_i|\le \eta t^{\lambda_i}\; (i=1, \, \dots, \, d-1), \text{ then } |\Psi(x') - \Psi(y')| \le t.
\end{align}
Let $|x_i-y_i|\le \eta t^{\lambda_i}$, $i=1, \, \dots, \, d-1$; i.e.,
\begin{align}
\label{max_xi_yi} \max _{1\le i\le d-1} \Bigl(\frac{|x_i-y_i|}{\eta}\Bigr)^{1/\lambda_i} \le t.
\end{align}
We estimate from above the value $|\Psi(x')-\Psi(y')|$. Without loss of generality, $\Psi(y')\ge \Psi(x')$.

Let $\varepsilon>0$. From \eqref{v_lambda}, \eqref{til_g_def_dop}, \eqref{psi_def_dop} it follows that there are a point $\xi = (\xi', \, \xi_d)\in U$ and a number $s_0\in (0, \, T]$ such that $\Psi(y')\le \xi_d - s_0 + s_0\eta + \varepsilon$, $y'\in \xi' + s_0^{\lambda'} [-\eta, \, \eta]^{d-1}$. Hence $s_0\ge \max _{1\le i\le d-1}(|\xi_i-y_i|/\eta)^{1/\lambda_i}$. Since $0<\eta<1$, we have
\begin{align}
\label{psi_s0}
\Psi(y')\le \xi_d - s + s\eta + \varepsilon, \quad \text{where }s = \max _{1\le i\le d-1}(|\xi_i-y_i|/\eta)^{1/\lambda_i}.
\end{align}

Let $\tilde s = \max _{1\le i\le d-1}(|\xi_i-x_i|/\eta)^{1/\lambda_i}$, and the maximum attains at the position $j$. Then $x'\in \xi' + \tilde s^{\lambda'}[-\eta, \, \eta]^{d-1}$. By \eqref{pi_sbs_xi}, we have $\tilde s\le T$. This together with \eqref{v_lambda}, \eqref{til_g_def_dop}, \eqref{psi_def_dop} yields that
\begin{align}
\label{psi_x_pr}
\Psi(x') \ge \xi_d -\tilde s+\tilde s\eta.
\end{align}
From the inequality $(a+b)^\beta \le a^\beta + b^\beta$ ($a, \, b\ge 0$, $0<\beta\le 1$) it follows that 
$$
\tilde s = \Bigl( \frac{|\xi_j-x_j|}{\eta}\Bigr)^{1/\lambda_j} \le \Bigl( \frac{|\xi_j-y_j|}{\eta}\Bigr)^{1/\lambda_j} + \Bigl( \frac{|y_j-x_j|}{\eta}\Bigr)^{1/\lambda_j} \stackrel{\eqref{max_xi_yi},\eqref{psi_s0}}{\le} s + t.
$$
Hence $$\Psi(y')-\Psi(x')\stackrel{\eqref{psi_s0},\eqref{psi_x_pr}}{\le} (\tilde s-s)(1-\eta) + \varepsilon \le (1-\eta)t +\varepsilon.$$
Since $\varepsilon>0$ is arbitrary, we get $\Psi(y')-\Psi(x') \le (1-\eta)t$. This completes the proof of \eqref{g_hol}.

From \eqref{podgr} and \eqref{g_hol} it follows that $\tilde G\in {\cal G}'_{\varphi_1, \, \dots, \, \varphi_d}$ with $\varphi_i(t) = \alpha \cdot t^{\lambda_i}$, $i=1, \, \dots, \, d-1$, where $\alpha\in (0, \, 1]$ is a constant.

We claim that $G \backslash \tilde G \subset V^{t_*}_\lambda$ for some $t_*>0$. This will complete the proof.

Let $(x', \, x_d)\in G\backslash \tilde G$. By \eqref{g_u_vl} and \eqref{til_g_def_dop}, we have $(x', \, x_d)\in U+V_\lambda$ and $x'\notin \Pi$. This together with \eqref{v_lambda} implies that there are a point $(\xi',\, \xi_d)\in U$ and a number $t\in (0, \, T]$ such that $(x', \, x_d)\in (\xi', \, \xi_d)-te_d + t^\lambda [-\eta, \, \eta]^d$, and, in addition, $|x_i|\ge T^{\lambda_i}\eta/5$ holds for some $i\in \{1, \, \dots, \, d-1\}$ (see \eqref{pi_def_dop}). Since $U\stackrel{\eqref{u_sbs_t}}{\subset} T^\lambda(-\eta/10, \, \eta/10)^d$, we have $|\xi_i|\le T^{\lambda_i}\eta/10$. Therefore, $|x_i|\le T^{\lambda_i}\eta/10 + t^{\lambda_i} \eta$. Hence $T^{\lambda_i}\eta/5 \le T^{\lambda_i}\eta/10 + t^{\lambda_i} \eta$; this implies that $t\ge T/10^{1/\lambda_i}\ge T/10$.


\begin{Biblio}
\bibitem{besov_holder} O. V. Besov, ``Estimates of entropy numbers of the Sobolev embedding operator on a H\"{o}lder domain'', {\it Mat. Zametki}, {\bf 118}:3 (2025), 366--379 (in Russian; English translation: to appear).

\bibitem{bricchi} M. Bricchi, ``Existence and properties of $h$-sets'', {\it Georgian Math. J.}, {\bf 9}:1 (2002), 13--32.

\bibitem{sob_book} S.L. Sobolev, {\it Applications of functional analysis in mathematical physics,}, M., Nauka, 1988.

\bibitem{resh1}  Yu. G. Reshetnyak, ``Integral representations of differentiable functions in domains
with nonsmooth boundary'', {\it Sib. Math. J.}, {\bf 21}:6 (1981), 833--839.

\bibitem{resh2} Yu. G. Reshetnyak, ``A remark on integral representations of differentiable
functions of several variables'', {\it Sib. Mat. Zh.}, {\bf 25}:5 (1984),  198--200 (in Russian).

\bibitem{bojar} B. Bojarski, ``Remarks on Sobolev imbedding inequalities'', Complex analysis (Joensuu, 1987), Lecture Notes in Math., 1351, Springer, Berlin, 1988, 52--68.

\bibitem{besov_rasp} O. V. Besov, ``Embedding of Sobolev spaces in domains with a decaying flexible cone condition'', {\it Proc. Steklov Inst. Math.}, {\bf 173} (1987), 13--30.

\bibitem{besov84} O. V. Besov, ``Integral representations of functions and embedding theorems for a domain with a flexible horn condition'', {\it Proc. Steklov Inst. Math.}, {\bf 170} (1987), 11--31.

\bibitem{labutin1} D. A. Labutin, ``Integral representations of functions and embeddings of Sobolev spaces on cuspidal domains'', {\it Math. Notes}, {\bf 61}:2 (1997), 164--179.

\bibitem{labutin2} D. A. Labutin, ``Embedding of Sobolev Spaces on H\"{o}lder Domains'', {\it Proc. Steklov Inst. Math.}, {\bf 227} (1999), 163--172.

\bibitem{trushin} B. V. Trushin, ``Sobolev Embedding Theorems for a Class of Anisotropic Irregular Domains'', {\it Proc. Steklov Inst. Math.}, 260 (2008), 287--309.

\bibitem{haj_kosk} P. Haj\l asz, P. Koskela, ``Isoperimetric inequalities and imbedding theorems in irregular domains'', {\it J. London Math. Soc.} (2), {\bf 58}:2 (1998), 425--450.

\bibitem{kilp_mal} T. Kilpel\"{a}inen, J. Mal\'{y}, ``Sobolev inequalities on sets with irregular boundaries'', {\it Z. Anal. Anwendungen}, {\bf 19}:2 (2000), 369--380.

\bibitem{maz_pob} V. G. Maz'ya, S. V. Poborchi, ``Imbedding theorems for Sobolev spaces on domains with peak and on H\"{o}lder domains'', {\it St. Petersburg Math. J.}, {\bf 18}:4 (2007), 583--605.

\bibitem{besov01} O. V. Besov, ``Sobolev's embedding theorem for a domain with irregular boundary'', {\it Sb. Math.}, {\bf 192}:3 (2001), 323--346.

\bibitem{besov10} O. V. Besov, ``Integral estimates for differentiable functions on irregular domains'', {\it Sb. Math.}, {\bf 201}:12 (2010), 1777--1790.

\bibitem{besov14} O. V. Besov, ``Embedding of Sobolev spaces and properties of the domain'', {\it Math. Notes}, {\bf 96}:3 (2014), 326--331.

\bibitem{besov15} O. V. Besov, ``Embedding of a Weighted Sobolev Space and Properties of the Domain'', {\it Proc. Steklov Inst. Math.}, {\bf 289} (2015), 96--103.

\bibitem{besov22} O. V. Besov, ``Conditions for Embeddings of Sobolev Spaces on a Domain with Anisotropic Peak'', {\it Proc. Steklov Inst. Math.}, {\bf 319} (2022), 43--55.

\bibitem{caso_ambr} L. Caso, R. D'Ambrosio, ``Weighted spaces and weighted norm inequalities on irregular domains'', {\it J. Approx. Theory}, {\bf 167} (2013), 42--58.

\bibitem{tikh60} V. M. Tikhomirov, ``Diameters of sets in function spaces and the theory of best approximations'', {\it Russian Math. Surveys}, {\bf 15}:3 (1960), 75--111.

\bibitem{babenko} K. I. Babenko, ``Approximation of periodic functions of many variables by trigonometric
polynomials'', {\it Dokl. Akad. Nauk SSSR} {\bf 132}:5 (1960), 247--250; English transl., Soviet
Math. Dokl. 1 (1960), 513--516.

\bibitem{mityagin} B. S. Mityagin, ``Approximation of functions in $L_p$ and $C$ spaces on the torus'', {\it Mat. Sb.} {\bf 58(100)}:4 (1962), 397--414. (Russian).

\bibitem{makovoz} Yu. I. Makovoz, ``On a method for estimation from below of diameters of sets in Banach spaces'', {\it Math. USSR-Sb.}, {\bf 16}:1 (1972), 139--146.

\bibitem{ismagilov} R.S. Ismagilov, ``Diameters of sets in normed linear spaces and the approximation of functions by trigonometric polynomials'',
{\it Russian Math. Surveys}, {\bf 29}:3 (1974), 169--186.

\bibitem{majorov} V. E. Maiorov, ``Discretization of the problem of diameters'', {\it Usp. Mat. Nauk}, {\bf 30}:6(186) (1975), 179--180 (in Russian).

\bibitem{bib_kashin} B.S. Kashin, ``The widths of certain finite-dimensional
sets and classes of smooth functions'', {\it Math. USSR-Izv.},
{\bf 11}:2 (1977), 317--333.

\bibitem{kashin_sma} B. S. Kashin, ``Widths of Sobolev classes of small-order smoothness'', Moscow Univ. Math. Bull., {\bf 36}:5 (1981), 62--66.

\bibitem{kulanin1} E. D. Kulanin, ``Estimates for diameters of Sobolev classes of small-order smoothness'', {\it Vestnik Mosk. Univ. Ser. 1. Matematika. Mekhanika}, 1983, no. 2,  24--30.

\bibitem{kulanin2} E. D. Kulanin, ``On the diameters of a class of functions of bounded variation in the space $L_q(0,\, 1)$, $2<q<\infty$'', {\it Russian Math. Surveys}, {\bf 38}:5 (1983), 146--147 .

\bibitem{galeev85} E.M. Galeev, ``Kolmogorov widths in the space $\widetilde{L}_q$ of the classes $\widetilde{W}_p^{\overline{\alpha}}$ and $\widetilde{H}_p^{\overline{\alpha}}$ of periodic functions of several variables'', {\it Math. USSR-Izv.}, {\bf 27}:2 (1986), 219--237.

\bibitem{teml3} V. N. Temlyakov, ``Approximation of periodic functions of several variables by
trigonometric polynomials, and widths of some classes of functions'', {\it Math. USSR-Izv.}, {\bf 27}:2 (1986), 285--322.

\bibitem{teml4} V. N. Temlyakov, ``Approximations of functions with bounded mixed derivative'', {\it Proc. Steklov Inst. Math.}, {\bf 178} (1989), 1--121.

\bibitem{galeev87} E.M. Galeev, ``Estimates for widths, in the sense of Kolmogorov, of classes of periodic functions of several variables with small-order smoothness'',
{\it Vestnik Moskov. Univ. Ser. I Mat. Mekh.} no. 1 (1987), 26--30 (in Russian).


\bibitem{mal25} Yu. V. Malykhin, ``Kolmogorov widths of the class $W_1^1$'', {\it Math. Notes}, {\bf 117}:6 (2025), 1034--1039. 


\bibitem{besov_cusp} O.V. Besov, ``Kolmogorov widths of Sobolev classes on an irregular domain'', {\it Proc. Steklov Inst. Math.} {\bf 280} (2013), 34--45.

\bibitem{vas_john} A. A. Vasil'eva, ``Widths of weighted Sobolev classes on a John domain'',
{\it Proc. Steklov Inst. Math.}, {\bf 280} (2013), 91--119.

\bibitem{evans_harris} W. D. Evans, D. J. Harris, ``Fractals, trees and the Neumann Laplacian'', {\it Math. Ann.}, {\bf 296}:3 (1993), 493--527.

\bibitem{evans_harris2} W. D. Evans, D. J. Harris, Y. Saito ``On the approximation numbers of Sobolev
embeddings on singular domains and trees'', {\it  Quart. J. Math.}, {\bf 55}:3 (2004), 267--302.

\bibitem{heinr} S. Heinrich, ``On the relation between linear $n$-widths and approximation numbers'', {\it J. Approx. Theory}, {\bf 58}:3 (1989), 315--333.

\bibitem{vas_cusp} A. A. Vasil'eva, ``Widths of Sobolev weight classes on a domain with cusp'', {\it Sb. Math.}, {\bf 206}:10 (2015), 1375--1409.

\bibitem{bir_sol} M. Sh. Birman, M. Z. Solomyak, ``Piecewise-polynomial approximations of functions of the classes $W^\alpha_p$'', {\it Math. USSR-Sb.}, {\bf 2}:3 (1967), 295--317.

\bibitem{edm_tr} D. E. Edmunds, H. Triebel, {\it Function spaces, entropy numbers, differential operators}, Cambridge Tracts in Math., 120, Cambridge Univ. Press, Cambridge, 1996.

\bibitem{schutt} C. Sch\"{u}tt, ``Entropy numbers of diagonal operators between symmetric Banach spaces'', {\it J. Approx. Theory}, {\bf 40}:2 (1984), 121--128.


\bibitem{vas_entr} A. A. Vasil'eva, ``Entropy numbers of embedding operators of function spaces on sets with tree-like structure'', {\it Izv. Math.}, {\bf 81}:6 (2017), 1095--1142.

\bibitem{edm89} D.E. Edmunds, R.M. Edmunds, ``Embeddings of anisotropic Sobolev spaces''.  Analysis and Continuum Mechanics. 1989. 269--276.

\bibitem{mattila} P. Mattila, {\it Geometry of Sets and Measures in Euclidean Spaces}, Cambridge Univ. Press, 1995.

\bibitem{vas_width} A. A. Vasil'eva, ``Widths of function classes on sets with tree-like structure'', {\it J. Approx. Theory}, {\bf 192} (2015), 19--59.

\bibitem{a_s} N. Arcozzi, R. Rochberg, E. Sawyer, ``Carleson measures for analytic Besov spaces'', {\it Rev. Mat. Iberoam.} {\bf 18}:2 (2002) 443--510.

\bibitem{vas_besov} A. A. Vasil'eva, ``Kolmogorov and linear widths of the weighted Besov classes with singularity at the origin'', {\it J. Approx. Theory}, {\bf 167} (2013), 1--41.

\bibitem{mazya_book} V. G. Maz'ya, {\it Sobolev Spaces} (Leningrad. Gos. Univ., Leningrad, 1985; Springer, Berlin, 1985).

\end{Biblio}
\end{document}